\documentclass{amsart}
\usepackage{amsfonts}
\usepackage{graphicx}
\usepackage{xcolor}

\newtheorem{theorem}{Theorem}
\theoremstyle{plain}

\newtheorem{corollary}{Corollary}

\newtheorem{definition}{Definition}

\newtheorem{lemma}{Lemma}
\newtheorem{notation}{Notation}

\newtheorem{proposition}{Proposition}

\numberwithin{equation}{section}
\theoremstyle{definition}
\newtheorem{example}{Example}
\newtheorem{remark}{Remark}

\DeclareMathOperator{\supp}{supp}
\DeclareMathOperator*{\esssup}{ess\,sup}
\DeclareMathOperator*{\essinf}{ess\,inf}

\begin{document}
\title[Dimensions of random Moran measures]{Assouad-like dimensions of a
class of random Moran measures II -- non-homogeneous Moran sets}
\author{Kathryn E. Hare}
\address{Dept. of Pure Mathematics, University of Waterloo, Waterloo, Ont.,
Canada, N2L 3G1}
\email{kehare@uwaterloo.ca}
\author{Franklin Mendivil}
\address{Department of Mathematics and Statistics, Acadia University,
Wolfville, N.S. Canada, B4P 2R6}
\email{franklin.mendivil@acadiau.ca}
\thanks{The research of K. Hare is partially supported by NSERC 2016:03719.
The research of F. Mendivil is partially supported by NSERC\ 2019:05237.}

\begin{abstract}
In this paper, we determine the almost sure values of the $\Phi$-dimensions
of random measures $\mu$ supported on random Moran sets in ${\mathbb{R}}^d$
that satisfy a uniform separation condition. This paper generalizes earlier
work done on random measures on homogeneous Moran sets \cite{HM} to the case
of unequal scaling factors. The $\Phi$-dimensions are intermediate
Assouad-like dimensions with the (quasi-)Assouad dimensions and the $\theta$%
-Assouad spectrum being special cases.

The almost sure value of $\dim_\Phi \mu$ exhibits a threshold phenomena,
with one value for ``large'' $\Phi$ (with the quasi-Assouad dimension as an
example of a ``large'' dimension) and another for ``small'' $\Phi$ (with the
Assouad dimension as an example of a ``small'' dimension). We give many
applications, including both where the scaling factors are fixed and the
probabilities are uniformly distributed and also where the probabilities are
fixed and the scaling factors are uniformly distributed. 
\end{abstract}

\maketitle



\section{Introduction}

A dimension provides a way of quantifying the size of a set. In the context
of subsets of a metric space, there are many different dimensions that have
been defined and each describes slightly different geometric properties of
the subset. Two well-known examples of this are the Hausdorff and
box-counting dimensions, which are both global measures of the geometry of
the given subset. It is also of substantial interest to understand the local
variation in the geometry and for this other dimensions have been introduced
including the (upper and lower) Assouad dimensions and variations. The
Assouad dimensions \cite{A2,Fraserbook,KLV,L1}, the less extreme
quasi-Assouad dimensions \cite{CDW,GH,LX}, the $\theta $-Assouad spectrum 
\cite{FY}, and (the most general of these) the intermediate Assouad-like $%
\Phi $-dimensions \cite{FY,GHM} all quantify various aspects of the
\textquotedblleft thickest\textquotedblright\ and \textquotedblleft
thinnest\textquotedblright\ parts of the set. These same Assouad-like
dimensions are all also available to quantify Borel measures on metric
spaces \cite{FH,HH,HHT}.

The $\Phi$-dimensions range between the box dimensions and the Assouad
dimensions and are also locally defined. However, they differ in the depth
of scales that they consider and thus can provide precise information about
the set or measure (see Section \ref{sec:phidim} for definitions). In this
paper we extend the investigation of the $\Phi$-dimensions of random $1$%
-variable measures on homogeneous Moran sets \cite{HM} to the case of random
measures supported on random Moran sets with multiple scaling factors for
the similarities.

The study of the dimensional properties of random fractal objects is
well-established, with some early papers investigating the almost sure
Hausdorff dimension \cite{Fa, Gr}, while more recently the Assouad and
related dimensions have also been investigated \cite{FMT,FT,GHM2,HM,Tr,Tr2}.

By a random Moran measure we mean a random Borel probability measure
supported on a random Moran construction in ${\mathbb{R}}^D$ (see Section %
\ref{sec:construction} for the precise details of the construction); our
construction can also be described as a random $1$-variable fractal measure.
The support of the measure is constructed by a random iterative procedure,
where at each stage we replace each component of the set with a random (but
uniformly bounded) number of randomly scaled, separated, compact, and
similar subsets. A random Borel probability measure is then defined on the
random limiting set by a similar iterative process which subdivides the
total mass by randomly choosing a set of probabilities at each step. The
process produces a $1$-variable fractal measure since at each level we make
one random choice and use that same choice for all subdivisions on that
level. Specifically, at level $n$ we choose $K_n$ random geometric scaling
factors for the similarities and $K_n$ random probabilities to use in
subdividing the mass and use these $2 K_n$ choices for every subdivision at
that level. This is in contrast with the stochastically self-similar (or $%
\infty$-variable) construction where the choice is made independently for
each subdivision. We make a blanket separation assumption which can be
thought of as a uniform strong convex separation condition.

For any dimension function $\Phi $, the $\Phi $-dimension of the resulting
random measure $\mu _{\omega }$ is almost surely constant and this value
depends on how $\Phi $ compares to the threshold function $\Psi (t)=\log
|\log t|/|\log t|$ near $0$; this behavior is similar to what was seen in 
\cite{GHM2,HM,Tr2}. For $\Phi \ll \Psi $\footnote{%
For $f,g>0,$ we will write $g\ll f$ if there is a function $A$ and $\delta >0
$ such that $f(t)\geq A(t)g(t)$ for all $0<t<\delta $ and $A(t)\rightarrow
\infty $ as $t\rightarrow 0^{+}$.} (the \textquotedblleft
small\textquotedblright\ dimension functions $\Phi $, such as the Assouad
dimension), the computations of the almost sure values of the upper and
lower $\Phi $-dimensions of the random measure $\mu _{\omega }$ are quite
similar to the homogeneous (same scaling factor for all children) case dealt
with in \cite{HM}. These computations involve the essential supremum (or
essential infimum) of ratios of the logarithm of a probability to the
logarithm of a scaling factor (see Section \ref{sec:dimthm_small}).
Furthermore, the almost sure value of the $\Phi $-dimension is the same for
all small dimension functions. It is natural to ask if there is a choice of
probabilities such the almost sure $\Phi $-dimension of the associated
random measures coincides with the almost sure $\Phi $-dimension of the
underlying sets, as is true in the homogenous case. In Section \ref%
{sec:small_dim_underlying} we show that this need not be true in the more
general situation.

In contrast, for $\Phi \gg \Psi $ (the \textquotedblleft
large\textquotedblright\ dimension functions, such as the quasi-Assouad
dimension), the computations are significantly different in the current
situation of different scaling factors. Roughly, the reason for this is that
the choice of the extremal branch down the tree of subdivisions depends on
what exponent (dimension) one thinks is the correct one. Thus the
computation of the $\Phi $-dimension involves solving an equation of the
form $G(\theta )=\theta $ to find the correct exponent. The function $G$ is
a ratio of expected values of logarithms of probabilities to logarithms of
scaling ratios (see Section \ref{subsec:mainthm} for details). Again, the
almost sure value of the $\Phi $-dimension is the same for all large
dimension functions. One special case we examine carefully is when the set
is deterministic with two scaling ratios, $a$ and $b$, and the probabilities
are uniformly chosen. Setting $a=b^{\gamma }$, the dimension is the root, $%
\theta $, of $b^{\theta }+b^{\gamma \theta }=e^{-1}$. Notice that this is a
polynomial in $b^{\theta }$ if $\gamma $ is an integer. It is interesting to
note that the dimension of the support (the Cantor-like set) in this case is
the root of $b^{\theta }+b^{\gamma \theta }=1$. Another special case we
examine is again when the set is deterministic, but now the
\textquotedblleft left\textquotedblright\ probability is chosen randomly
from the two possibilities $p$ or $1-p$ (for a fixed value of $p$). In this
case the almost sure $\Phi $-dimension of $\mu _{\omega }$ is given
explicitly as one of two values where the one to use depends on the
relationship between $a$ and $b$ and also between $p$ and $1-p$. All of
these examples are discussed in Section \ref{subsec:Cab}. It is an open
problem if the probabilities can be chosen so that the almost sure $\Phi $%
-dimension of the random measures coincides with that of the random sets. 

The definition and basic properties of the $\Phi$-dimensions are given in
Section \ref{sec:phidim} and the details of the random construction are
given in Section \ref{sec:construction}. Section \ref{sec:LargePhi} contains
our results for large $\Phi$ and Section \ref{sec:SmallPhi} those for small $%
\Phi$.

We present most of our discussion in the context of random subsets of ${%
\mathbb{R}}$ where at each stage we split each component into two
``children''. This is done for simplicity of exposition only and in Section %
\ref{sec:multi} we briefly indicate what changes are necessary to
accommodate random subsets of ${\mathbb{R}}^D$ with a random (but uniformly
bounded) number of children at each level.

It is important to note that we always assume that the scaling ratios are
uniformly bounded away from $0$. It is certainly possible to remove this
assumption, but this seems to require some delicate technical arguments and
we leave this case for future work.

\section{Assouad-like dimensions}

\label{sec:phidim}

There are many ways to quantify the `size' of subsets of metric spaces and
Borel probability measures on these metric spaces. The so-called $\Phi $%
-dimensions provide refined information on the local size of a set or
concentration of a measure. To define these, we first recall some standard
notation and define what we mean by a dimension function.

\begin{notation}
We will write$\ B(x,R)$ for the open ball centred at $x$ belonging to the
bounded metric space $X$ and radius $R$. By $N_{r}(E)$ we mean the least
number of open balls of radius $r$ required to cover $E\subseteq X$.
\end{notation}

\begin{definition}
A \textbf{dimension function} is a map $\Phi :(0,1)\rightarrow \mathbb{R}%
^{+} $ with the property that $\Psi(t) = t^{1+\Phi (t)}$ decreases to $0$ as $t$
decreases to $0$ and is  doubling, that is, there are $C_1, C_2 > 0$ so that 
$\Psi(2 t) \le C_1 \Psi(t)$ for all $0 < t < C_2$.
\end{definition}

Examples include the constant functions $\Phi (t)=\delta \geq 0,$ the
function $\Phi (t)=1/|\log t|$ and the function $\Phi (t)=\log \left\vert
\log t\right\vert /\left\vert \log t\right\vert $. The latter will be of
particular interest in this paper.

\begin{definition}
We will say that a dimension function $\Phi $ is \textbf{large} if 
\begin{equation*}
\Phi (t)=H(t)\frac{\log \left\vert \log t\right\vert }{\left\vert \log
t\right\vert }
\end{equation*}
where $H(t)\rightarrow \infty $ as $t\rightarrow 0$ and \textbf{small }if
(with the same notation) $H(t)\rightarrow 0$ as $t\rightarrow 0$.
\end{definition}

\begin{definition}
Let $\mu $ be a measure on $X$ and $\Phi $ be a dimension function. The 
\textbf{upper and lower }$\Phi $\textbf{-dimensions} of $\mu $ are given,
respectively, by%
\begin{equation*}
\overline{\dim }_{\Phi }\mu =\inf \left\{ 
\begin{array}{c}
d:(\exists C_{1},C_{2}>0)(\forall 0<r<R^{1+\Phi (R)}\leq R\leq C_{1}) \\ 
\frac{\mu (B(x,R))}{\mu (B(x,r))}\leq C_{2}\left( \frac{R}{r}\right) ^{d}%
\text{ }\forall x\in \supp\mu%
\end{array}%
\right\}
\end{equation*}%
and 
\begin{equation*}
\underline{\dim }_{\Phi }\mu =\sup \left\{ 
\begin{array}{c}
d:(\exists C_{1},C_{2}>0)(\forall 0<r<R^{1+\Phi (R)}\leq R\leq C_{1}) \\ 
\frac{\mu (B(x,R))}{\mu (B(x,r))}\geq C_{2}\left( \frac{R}{r}\right) ^{d}%
\text{ }\forall x\in \supp \mu%
\end{array}%
\right\} .
\end{equation*}
\end{definition}

These dimensions were introduced in \cite{HH} and were motivated, in part,
by the $\Phi $-dimensions of sets, introduced in \cite{FY} and thoroughly
studied in \cite{GHM}. We recall the definition.

\begin{definition}
The \textbf{upper }and \textbf{lower }$\Phi $\textbf{-dimensions\ }of $%
E\subseteq X$ are given, respectively, by 
\begin{equation*}
\overline{\dim }_{\Phi }E=\inf \left\{ 
\begin{array}{c}
d:(\exists C_{1},C_{2}>0)(\forall 0<r\leq R^{1+\Phi (R)} < R<C_{1})\text{ }
\\ 
N_{r}(B(z,R)\bigcap E)\leq C_{2}\left( \frac{R}{r}\right) ^{d}\text{ }%
\forall z\in E%
\end{array}%
\right\}
\end{equation*}%
and%
\begin{equation*}
\underline{\dim }_{\Phi }E=\sup \left\{ 
\begin{array}{c}
d:(\exists C_{1},C_{2}>0)(\forall 0<r\leq R^{1+\Phi (R)} < R<C_{1})\text{ }
\\ 
N_{r}(B(z,R)\bigcap E)\geq C_{2}\left( \frac{R}{r}\right) ^{d}\text{ }%
\forall z\in E%
\end{array}%
\right\} .
\end{equation*}
\end{definition}

\begin{remark}
(i) In the special case of $\Phi =0$, these dimensions are known as the 
\textbf{upper }and \textbf{lower Assouad dimensions }of the measure or set.
For measures, these dimensions are also known as the upper and lower
regularity dimensions and were studied by K\"{a}enm\"{a}ki et al in \cite%
{KL, KLV} and Fraser and Howroyd in \cite{FH}. The upper and lower Assouad
dimensions of the measure $\mu $ are denoted $\dim _{A}\mu $ and $\dim
_{L}\mu $ respectively and are important because the measure $\mu $ is
doubling if and only if $\dim _{A}\mu $ $<\infty $ (\cite{FH}) and uniformly
perfect if and only if $\dim _{L}\mu >0$ (\cite{KL}).

(ii) If we put $\Phi _{\theta }=1/\theta -1$ for $0<\theta <1,$ then $%
\overline{\dim }_{\Phi _{\theta }}\mu $ and \underline{$\dim $}$_{\Phi
_{\theta }}\mu $ are (basically) the \textbf{upper} and \textbf{lower $%
\theta $-Assouad spectrum} introduced in \cite{FY}. The \textbf{upper} and 
\textbf{lower quasi-Assouad dimensions} of $\mu ,$ developed in \cite{HHT,
HT}, are given by 
\begin{equation*}
\dim _{qA}\mu =\lim_{\theta \rightarrow 1}\overline{\dim }_{\Phi _{\theta
}}\mu \text{ and }\dim _{qL}\mu =\lim_{\theta \rightarrow 1}\underline{\dim }%
_{\Phi _{\theta }}\mu .
\end{equation*}
\end{remark}

Here are some basic relationships between these dimensions; for proofs see 
\cite{FY}, \cite{GHM}, \cite{HH} and the references cited there.

\begin{proposition}
{\ } Let $\Phi ,\Psi $ be dimension functions and $\mu $ be a measure.

(i) If $\Phi (t)\leq \Psi (t)$ for all $t>0$, then $\overline{\dim }_{\Psi
}\mu $ $\leq \overline{\dim }_{\Phi }\mu $ and $\underline{\dim }_{\Phi }\mu
\leq \underline{\dim }_{\Psi }\mu $.

(ii) We have that 
\begin{equation*}
\dim _{A}\mu \geq \overline{\dim }_{\Phi }\mu \geq \overline{\dim }_{\Phi }\ %
\supp\mu \geq \dim _{H}\supp \mu
\end{equation*}
and $\dim _{L}\mu \leq \underline{\dim }_{\Phi }\mu $. If $\mu $ is
doubling, then $\underline{\dim }_{\Phi }\mu \leq \underline{\dim }_{\Phi }%
\supp \mu . $

(iii) If $\Phi (t)\rightarrow 0$ as $t\rightarrow 0,$ then $\underline{\dim }%
_{\Phi }\mu \leq \dim _{qL}\mu $ and $\dim _{qA}\mu \leq \overline{\dim }%
_{\Phi }\mu $ .

(iv) If $\Phi (t)\leq 1/\left\vert \log t\right\vert $ for $t$ near $0$,
then $\overline{\dim }_{\Phi }\mu =\dim _{A}\mu $ and $\underline{\dim }%
_{\Phi }\mu =\dim _{L}\mu $.

(v) For any set $E,$ 
\begin{equation*}
\dim _{L}E\leq \underline{\dim }_{\Phi }E\leq \underline{\dim }_{B}E\leq 
\overline{\dim }_{B}E\leq \overline{\dim }_{\Phi }E\leq \dim _{A}E.
\end{equation*}%
(Here $\underline{\dim }_{B}$ and $\overline{\dim }_{B}$ are the lower and
upper box dimensions.)
\end{proposition}

\section{Random Moran sets and Measures}

\label{sec:construction}

\subsection{Definition of random Moran sets $\mathcal{C}_{\protect\omega }$
and random measures $\protect\mu_\protect\omega$}

For the majority of this paper we describe our results in the simple context
of subsets of $[0,1]$ with two ``children'' at each ``level''. We do this
for clarity and to highlight the important features of the construction.
However, in Section \ref{sec:multi} we briefly indicate the natural
extension to compact subsets of ${\mathbb{R}}^D$ with an arbitrary (but
uniformly bounded) number of children at each level. All of our proofs are
given so that they can be easily modified for the more general situation.

Let $(\Omega ,\mathcal{P})$ be a probability space. Fix $0< 2A \le B<1$ and
choose independently and identically distributed random variables 
\begin{equation*}
(a_{n}(\omega ),b_{n}(\omega ), p_n(\omega))\in \{(x,y,z)\in [
0,1]^{3}:A\leq \min (x,y)<x+y\leq B\}.
\end{equation*}%
We assume that $\mathbb{E}(e^{-t\log p_{n}}) = \mathbb{E}( p_n^{-t}) <\infty 
$ and $\mathbb{E}(e^{-t\log (1-p_{n})})= \mathbb{E}( (1-p)^{-t}) < \infty$
for some $t>0$. Note that this implies that the probability that $p_n = 0$
or $p_n = 1$ is zero. Note also that since $A > 0$, we have $\mathbb{E}%
(e^{-t\log a_{n}})=\mathbb{E}(a_{n}^{-t})<\infty $ and $\mathbb{E}(e^{-t\log
b_{n}}) = \mathbb{E}(b_n^{-t})<\infty $ for all $t>0$.

Let $L$ denote the minimal positive integer such that 
\begin{equation}
2B^{L-1}\leq 1-B.  \label{Ldefn}
\end{equation}

To create the random Moran set, $\mathcal{C}_{\omega },$ we begin with the
closed interval $[0,1]$ and then at step one form the set $\mathcal{C}%
_{\omega }^{(1)}$ by keeping the outer-most left subinterval of length $%
a_{1}(\omega )$ and the outer-most right subinterval of length $b_{1}(\omega
)$. Having inductively created $\mathcal{C}_{\omega }^{(n-1)},$ a union of $%
2^{n-1}$ closed intervals $\{I_{j}(\omega )\}_{j=1}^{2^{n-1}}$ (which we
call the Moran intervals of step (or level) $n-1$), we let $\mathcal{C}%
_{\omega }^{(n)}=\bigcup_{j=1}^{2^{n-1}}(I_{j}^{(1)}\cup I_{j}^{(2)})$ where 
$I_{j}^{(1)}=I_{j}^{(1)}(\omega )$ is the outer-most left closed subinterval
of $I_{j}=I_{j}(\omega )$ of length $|I_{j}^{(1)}|=a_{n}(\omega )|I_{j}|$
and $I_{j}^{(2)}=I_{j}^{(2)}(\omega )$ is the outer-most right closed
subinterval of $I_{j}$ of length $|I_{j}^{(2)}|=b_{n}(\omega )|I_{j}|$. We
call $I_{j}^{(1)}$ the left child of $I_{j}$ and $I_{j}^{(2)},$ the right
child. The \textbf{random Moran set }$\mathcal{C}_{\omega }$ is the compact
set 
\begin{equation*}
\mathcal{C}_{\omega }=\bigcap_{n=1}^{\infty }\mathcal{C}_{\omega }^{(n)}.
\end{equation*}

It can be convenient to label the Moran intervals of step $N$ as $%
I_{v_{1}\cdot \cdot \cdot v_{N}}$ with $v_{j}\in \{0,1\},$ where $%
I_{v_{1}\cdot \cdot \cdot v_{N-1}0}$ is the left child of $I_{v_{1}\cdot
\cdot \cdot v_{N-1}}$ and $I_{v_{1}\cdot \cdot \cdot v_{N-1}1}$ is the right
child. When we write $I_{N}(x)$ we mean the Moran interval of step $N$
containing the element $x\in \mathcal{C}_{\omega }$.

Notice that any Moran interval of step $N$ has length between $A^{N}$ and $%
B^{N}$ and 
\begin{equation*}
A^{k}\leq \frac{\left\vert I_{N+k}(x)\right\vert }{\left\vert
I_{N}(x)\right\vert }\leq B^{k}
\end{equation*}%
for any $N,x$. In particular this means that none of the intervals disappear.

The \textbf{random measure} $\mu _{\omega }$ is defined by the rule that $%
\mu _{\omega }([0,1])=1$ and if $I_{N}$ is a Moran interval of step $N,$
then (with the notation as above) 
\begin{equation*}
\mu _{\omega }(I_{N}^{(1)})=p_{N+1}(\omega )\mu _{\omega }(I_{n})\text{ and }%
\mu _{\omega }(I_{N}^{(2)})=(1-p_{N+1}(\omega ))\mu _{\omega }(I_{N}).
\end{equation*}%
For each $\omega $, this uniquely determines a probability measure on $%
\mathcal{C}_{\omega }$. In addition, for almost all $\omega $ the support is
all of $\mathcal{C}_{\omega }$. For those familiar with $V$-variable
fractals (see \cite{Betc}), we mention that our construction produces a
random $1$-variable fractal measure. Our entire random model can also be
viewed as sampling from the product space 
\begin{equation*}
\Omega = \prod_{n=1}^\infty \biggl ( \{ (x,y,z) : A \le \min\{x,y\} \le x+y
\le B, 0 \le z \le 1 \} \biggr ),
\end{equation*}
where we use the product measure on $\Omega$ induced by a given probability
measure on each factor.

Notice that we allow the possibility that $a_n$, $b_n$ and $p_n$ can be
dependent or independent of each other; we only assume that $(a_n,b_n,p_n)$
is independent of $(a_m, b_m, p_m)$ when $n \ne m$.

We remark that $\mathcal{C}_{\omega }$ has a \textquotedblleft uniform
separation\textquotedblright\ property in the sense that the distance
between the two children of $I_{N}$ is at least $(1-B)|I_{N}|$. This fact
allows us to prove the following simple, but useful, relationship between
Moran intervals of various levels and balls.

\begin{lemma}
\label{L1} Given $\omega \in \Omega ,$ $x\in \mathcal{C}_{\omega }$ and $%
0<R<1,$ choose the integer $N=N(\omega ,x)$ such that $\left\vert
I_{N}(x)\right\vert \leq R<\left\vert I_{N-1}(x)\right\vert $. Then 
\begin{equation*}
I_{N}(x)\cap \mathcal{C}_{\omega }\subseteq B(x,R)\cap \mathcal{C}_{\omega
}\subseteq I_{N-L}(x)\cap \mathcal{C}_{\omega }\text{.}
\end{equation*}
\end{lemma}

\begin{proof}
The proof is similar to \cite[Lemma 1]{HM}, but we include it here for
completeness. Obviously, $I_{N}(x)$ is contained in $B(x,R)$.

Assume $I_{N}^{\prime }$ is another Moran interval of step $N$ which
intersects $B(x,R)$ and suppose $I_{N-k}(x)$ is the common ancestor of $%
I_{N}(x)$ and $I_{N}^{\prime }$ with $k$ minimal. Then the two level $N$
intervals $I_{N}(x)$ and $I_{N}^{\prime }$ must be separated by a distance
of at least $|I_{N-k}(x)|(1-B)$ and at most $2R$. If $k\geq L,$ the
definition of $L$ gives%
\begin{eqnarray*}
|I_{N-k}(x)|(1-B) &\leq &2R<2|I_{N-1}(x)|\leq 2B^{k-1}|I_{N-k}(x)| \\
&\leq &2B^{L-1}\left\vert I_{N-k}(x)\right\vert \leq |I_{N-k}(x)|(1-B),
\end{eqnarray*}%
which is a contradiction. Hence, all step $N$ Moran intervals intersecting $%
B(x,R)$ are contained in $I_{N-L}(x)$ and that implies $B(x,R)\bigcap 
\mathcal{C}_{\omega }\subseteq I_{N-L}(x)$.
\end{proof}

Our next lemma shows that the dimension of $\mu_\omega$ is completely
determined by the lengths and measures of the Moran intervals. While this
result is not surprising because of our separation assumption, it is very
useful to make it explicit.

\begin{lemma}
\label{lem:intervals_give_dimension} Let 
\begin{equation*}
\Delta _{\omega }=\inf \left\{ 
\begin{array}{c}
d:(\exists c_{1},c_{2}>0)(\forall \ I_{n}(\omega )\subseteq I_{N}(\omega
),\left\vert I_{N}\right\vert \leq c_{1},\left\vert I_{n}\right\vert
<\left\vert I_{N}\right\vert ^{1+\Phi (\left\vert I_{N}\right\vert )})\text{ 
} \\ 
\frac{\mu _{\omega }(I_{N})}{\mu _{\omega }(I_{n})}\leq c_{2}\left( \frac{%
\left\vert I_{N}\right\vert }{\left\vert I_{n}\right\vert }\right) ^{d}\text{
}%
\end{array}%
\right\} 
\end{equation*}
and
\begin{equation*}
\delta _{\omega }=\sup \left\{ 
  \begin{array}{c}
        d:(\exists c_{1},c_{2}>0)(\forall \ I_{n}(\omega )\subseteq I_{N}(\omega),
                   \left\vert I_{N}\right\vert \leq c_{1},\left\vert I_{n}\right\vert
           <\left\vert I_{N}\right\vert ^{1+\Phi (\left\vert I_{N}\right\vert )})\  \\ 
          \frac{\mu _{\omega }(I_{N})}{\mu _{\omega }(I_{n})}\geq 
               c_{2}\left( \frac{\left\vert I_{N}\right\vert }{\left\vert I_{n}\right\vert }\right) ^{d}
\end{array}
\right\}.
\end{equation*} 
Then $\Delta _{\omega }=\overline{\dim }_{\Phi }\mu _{\omega }$ and
$\delta _{\omega }=\underline{dim}_\Phi \mu _{\omega }$. 
\end{lemma}

\begin{proof}
We fix an $\omega \in \Omega$ for the rest of the proof and simplify our
notation by removing any explicit mention of the dependence on $\omega$.

Let $\varepsilon >0$ and get constants $c_{1},c_{2}$ such that 
\begin{equation*}
\frac{\mu (I_{N})}{\mu (I_{n})}\leq c_{2}\left( \frac{\left\vert
I_{N}\right\vert }{\left\vert I_{n}\right\vert }\right) ^{\Delta
+\varepsilon }
\end{equation*}%
whenever $I_{n}\subseteq I_{N}$ with $\left\vert \text{ }I_{N}\right\vert
\leq c_{1}$ and $\left\vert I_{n}\right\vert <\left\vert I_{N}\right\vert
^{1+\Phi (\left\vert I_{N}\right\vert )}$. Choose $N_{0}$ so that all Moran
intervals of level $N_{0}-L$ have diameter at most $c_{1}$. Choose $x\in 
\mathcal{C}_{\omega },$ and suppose $R\leq A^{N_{0}}$ and $0<r<R^{1+\Phi
(R)} $. Obtain $n\geq N\geq N_{0}$ such that 
\begin{equation*}
\left\vert I_{N}(x)\right\vert \leq R<\left\vert I_{N-1}(x)\right\vert \leq
\left\vert I_{N-L}(x)\right\vert \leq c_{1}\text{ and }\left\vert
I_{n}(x)\right\vert \leq r<\left\vert I_{n-1}(x)\right\vert .
\end{equation*}%
By Lemma \ref{L1}, $B(x,r)\supseteq I_{n}(x)$ and $B(x,R)\cap \mathcal{C}%
_{\omega }\subseteq I_{N-L}(x)$.

As the function $t^{1+\Phi (t)}$ is decreasing as $t\downarrow 0$, $%
\left\vert I_{n}(x)\right\vert \leq r<R^{1+\Phi (R)}\leq \left\vert
I_{N-L}\right\vert ^{1+\Phi (\left\vert I_{N-L}\right\vert )}$. Hence%
\begin{equation*}
\frac{\mu (B(x,R))}{\mu (B(x,r))}\leq \frac{\mu (I_{N-L})}{\mu (I_{n})}\leq
c_{2}\left( \frac{\left\vert \text{ }I_{N-L}\right\vert }{\left\vert
I_{n}\right\vert }\right) ^{\Delta +\varepsilon }\leq c_{2}\left( \frac{%
A^{-L}\left\vert I_{N}\right\vert }{A\left\vert I_{n-1}\right\vert }\right)
^{\Delta +\varepsilon }\leq C_{2}\left( \frac{R}{r}\right) ^{\Delta
+\varepsilon }
\end{equation*}%
for $C_{2}=c_{2}A^{-(L+1)(\Delta +\varepsilon )}$ and consequently, $%
\overline{\dim }_{\Phi }\mu \leq \Delta $.

The opposite inequality is similar. Let $D=\overline{\dim }_{\Phi }\mu $ and
given $\varepsilon >0$ choose $C_{1},C_{2}$ such that%
\begin{equation*}
\frac{\mu (B(x,R))}{\mu (B(x,r))}\leq C_{2}\left( \frac{R}{r}\right)
^{D+\varepsilon }
\end{equation*}%
whenever $r<R^{1+\Phi (R)}\leq R\leq C_{1}$ and $x\in \mathcal{C}_{\omega }$%
. Suppose that $I_{n}\subseteq I_{N}$ with $\left\vert \text{ }%
I_{N}\right\vert \leq C_{1}$ and $\left\vert I_{n}\right\vert <\left\vert
I_{N}\right\vert ^{1+\Phi (\left\vert I_{N}\right\vert )}$. Choose $x\in 
\mathcal{C}_{\omega }$ such that $I_{n}=I_{n}(x)$ and $I_{N}=I_{N}(x)$. Let $%
R=\left\vert I_{N}(x)\right\vert \leq C_{1}$ and $r=\left\vert
I_{n}(x)\right\vert (1-B)$. As the distance from $I_{n}(x)$ to the nearest
Moran interval of level $n$ is at most $r,$ $B(x,r)\cap \mathcal{C}_{\omega
}\subseteq I_{n}(x)$. Clearly $B(x,R)\supseteq I_{N}(x)$ and $r<R^{1+\Phi
(R)}$. Thus 
\begin{equation*}
\frac{\mu (I_{N})}{\mu (I_{n})}\leq \frac{\mu (B(x,R))}{\mu (B(x,r))}\leq
C_{2}\left( \frac{R}{r}\right) ^{D+\varepsilon }=C_{2}(1-B)^{-(D+\varepsilon
)}\left( \frac{|I_{N}|}{\left\vert I_{n}\right\vert }\right) ^{D+\varepsilon
},
\end{equation*}%
which proves $\Delta \leq D$.

\medskip

The proof of the second part involving the lower dimension is fairly different.
Let $\Psi(t) = t^{1 + \Phi(t)}$ and
choose $\eta $ from the doubling property so that $\Psi (x)\leq \eta \Psi
((1-B)x)$ and $\Psi (x)\leq \eta \Psi (Ax)$. 
Choose $T=T(\eta )$ such that  $B^{T}\leq A^{L}/\eta$.

Let $d=\underline{\dim }_{\Phi }\mu$. 
Fix $\varepsilon >0$ and small $c_{2}>0$. 
There must be $x,R<A^{L+1}$ and $r\leq \Psi (R)$ such that 
\begin{equation*}
  \frac{\mu (B(x,R))}{\mu (B(x,r))}\leq 
      c_{2}\left( \frac{R}{r}\right)^{d+\varepsilon }.
\end{equation*}%
Choose $n,N>L$ such that 
\begin{eqnarray*}
   \left\vert I_{N}(x)\right\vert  &\leq &R\leq 
    \left\vert I_{N-1}(x)\right\vert  \\
        \left\vert I_{n}(x)\right\vert  &\leq &r\leq \left\vert I_{n-1}(x)\right\vert .
\end{eqnarray*}%
Then 
\begin{eqnarray*}
    I_{N}(x) &\subseteq &B(x,R)\cap C_{\omega }\subseteq I_{N-L}(x)\text{ and} \\
    I_{n}(x) &\subseteq &B(x,r)\cap C_{\omega }\subseteq I_{n-L}(x).
\end{eqnarray*}%
As well, $\Psi (R)\leq \Psi (|I_{N}|/A)\leq \eta \Psi (|I_{N}|)$. 
Note that we can assume $c_{2}$ is so small that $n>N+L$.

Since $B(x,r)\subseteq I_{n-L}(x)$ it is covered by at most $2^{T}$ Cantor
intervals at level $n-L+T,$ say the intervals $I_{n-L+T}(y_{j})$ with 
$y_{j}\in B(x,r)\cap I_{n-L}(x)$. 
For each $j$, 
\begin{equation*}
   \left\vert I_{n-L+T}(y_{j})\right\vert \leq B^{T}\left\vert
     I_{n-L}(x)\right\vert \leq \frac{B^{T}}{A^{L}}\left\vert I_{n}(x)\right\vert
      \leq \frac{1}{\eta }r\leq \frac{1}{\eta }\Psi (R)\leq 
    \frac{1}{\eta }\Psi\left( \frac{\left\vert I_{N}\right\vert }{A}\right) 
   \leq \Psi (|I_{N}|).
\end{equation*}%
Moreover, since $\mu (B(x,r))$ is at most the sum of the measures of these
Cantor intervals, one of them, say $I_{n-L+T}(y)$ with 
$y\in B(x,r)\cap I_{n-L}(x)$, satisfies
\begin{equation*}
    \mu (I_{n-L+T}(y))\geq 2^{-T}\mu (B(x,r)).
\end{equation*}%
Lastly, we note that as $y\in I_{n-L}(x)\subseteq I_{N}(x),$ we have 
$I_{N}(y)=I_{N}(x)\subseteq B(x,R)$.

Putting together these facts, we see that 
\begin{eqnarray*}
   \frac{\mu (I_{N}(y))}{\mu (I_{n-L+T}(y))} &\leq 
   &\frac{\mu (B(x,R))}{2^{-T}\mu (B(x,r))}
          \leq 2^{T}c_{2}\left( \frac{R}{r}\right) ^{d+\varepsilon }
\\
    &\leq &2^{T}c_{2}\left( \frac{\frac{1}{A}\left\vert I_{N}(y)\right\vert }
          {\eta \left\vert I_{n-L+T}(y)\right\vert }\right) ^{d+\varepsilon}=
       C_{2}\left( \frac{\left\vert I_{N}(y)\right\vert }{\left\vert I_{n-L+T}(y)\right\vert }\right)^{d \varepsilon }
\end{eqnarray*}%
where the new constant $C_{2}$ can be made as small as desired by choosing $c_{2}$ suitably small.

Since we can do this for arbitrarily small $\varepsilon ,C_{2}$ and $|I_{N}|$
it follows that $\delta \leq d$.

The argument for the reverse inequality is similar.

\end{proof}

 Notice that we used the doubling assumption on $t \to t^{1 + \Phi(t)}$ only for the lower dimension and so all the results in our paper for the upper dimension hold without this assumption.

Using this lemma it is simple to show that the $\Phi$-dimensions of $\mu_\omega$ are almost surely constant.

\begin{proposition}
\label{prop:almostsuredim} For any dimension function $\Phi$, the upper and
lower $\Phi$-dimensions are almost surely constant functions of $\omega$.
\end{proposition}

\begin{proof}
We show that $\omega \mapsto \dim _{\Phi }\mu _{\omega }$ is a permutable
random variable (meaning that it is invariant under any finite permutation
of the levels) and thus is almost surely constant by the Hewit-Savage
zero-one law \cite{Chung}. To see this, let $\omega $ be fixed and $\pi :{%
\mathbb{N}}\rightarrow {\mathbb{N}}$ be a permutation that fixes all but
finitely many values. Suppose that $N_{0}$ is the largest such value. We use 
$I$ to denote a Moran interval from the unpermuted construction and $J$ for
a Moran interval from the permuted construction. Then for any $n>N_{0}$ and
choice $v_{1},v_{2},\ldots ,v_{n}\in \{0,1\}$, it is clear from the
description of the construction that $|I_{v_{1}v_{2}\cdots
v_{n}}|=|J_{v_{1}v_{2}\ldots v_{n}}|$. Thus the proposition follows from
Lemma \ref{lem:intervals_give_dimension}.
\end{proof}

\section{\protect\bigskip Dimension results for large $\Phi $}

In this section we continue to use the notation and assumptions from Section %
\ref{sec:construction}.

\label{sec:LargePhi}

\subsection{Statement of the dimension theorem for large $\Phi$ and
preliminary results}

\label{subsec:mainthm}

When computing the $\phi $-dimension of $\mu $ we need to compare ratios of
lengths to ratios of mass under $\mu $ (as in equation (\ref{eq:YandZ})).
The definitions of the random variables $Y$, $Z$, and $G$ (given next) can
be understood using this, as will be clear from the work in this section.

\begin{notation}
\smallskip Given $\theta \ge 0$, we define the iid random variables $%
Y_{n}(\theta), Z_{n}(\theta): \Omega \to {\mathbb{R}}$ by 
\begin{equation*}
Y_{n}(\theta)(\omega)=\left\{ 
\begin{array}{cc}
\log p_{n}(\omega ) & \text{if }p_{n}(\omega )\leq \frac{a_{n}^{\theta
}(\omega )}{a_{n}^{\theta }(\omega )+b_{n}^{\theta }(\omega )} \\ 
\log (1-p_{n}(\omega )) & \text{if }p_{n}(\omega )>\frac{a_{n}^{\theta
}(\omega )}{a_{n}^{\theta }(\omega )+b_{n}^{\theta }(\omega )}%
\end{array}%
\right.
\end{equation*}%
and 
\begin{equation*}
Z_{n}(\theta)(\omega)=\left\{ 
\begin{array}{cc}
\log a_{n}(\omega ) & \text{if }p_{n}\leq \frac{a_{n}^{\theta }(\omega )}{%
a_{n}^{\theta }(\omega )+b_{n}^{\theta }(\omega )} \\ 
\log b_{n}(\omega ) & \text{if }p_{n}>\frac{a_{n}^{\theta }(\omega )}{%
a_{n}^{\theta }(\omega )+b_{n}^{\theta }(\omega )}%
\end{array}%
\right. .
\end{equation*}%
Random variables $Y_{n}^{\prime },Z_{n}^{\prime }$ are defined similarly,
but with the relationship between $p_{n}$ and $\frac{a_{n}^{\theta }}{%
a_{n}^{\theta }+b_{n}^{\theta }}$ interchanged. Put 
\begin{equation}
G(\theta )=\frac{\mathbb{E}_{\omega }\mathbb{(}Y_{1}(\theta)(\omega))}{%
\mathbb{E}_{\omega }\mathbb{(}Z_{1}(\theta)(\omega))}\text{ and }
G^{\prime}(\theta )= \frac{\mathbb{E}_{\omega }\mathbb{(}Y_{1}^{\prime
}(\theta)(\omega))}{\mathbb{E}_{\omega }\mathbb{(}Z_{1}^{\prime
}(\theta)(\omega))}\text{.}  \label{G}
\end{equation}
\end{notation}

We have written $\mathbb{E}_{\omega }$ to emphasize that the expectation is
taken over the variable $\omega $.

The condition $p\leq \frac{a^{\theta }}{a^{\theta }+b^{\theta }}$ is
relevant because it is equivalent to $a^{\theta }/p\geq b^{\theta }/(1-p)$,
an inequality very important for computing these dimensions.

It would be interesting to explore the properties of the functions $G(\theta
)$ and $G^{\prime }(\theta )$, and better understand them as objects in
their own right. However, in this paper we mainly view these functions as
technical tools that we use in our proofs. We do provide some discussion in
Section \ref{subsubsec:G_theta} and plots for a few examples in the
Appendix. We have also explicitly computed $G(\theta )$ in a few of the
examples in Section \ref{subsec:Cab}.

With this notation we can now state our main result for large dimension
functions $\Phi $.

\begin{theorem}
\label{MainLarge} (i) Suppose $G(\psi )<\psi $. There is a set $\Gamma_\psi
\subseteq \Omega ,$ of full measure in $\Omega ,$ such that 
\begin{equation*}
\overline{\dim }_{\Phi }\mu _{\omega }\leq \psi \text{ }
\end{equation*}%
for all large dimension functions $\Phi $ and all $\omega \in \Gamma_\psi $.

(ii) Suppose $G(\psi )\geq \psi $. There is a set $\Gamma_\psi \subseteq
\Omega ,$ of full measure in $\Omega ,$ such that%
\begin{equation*}
\overline{\dim }_{\Phi }\mu _{\omega }\geq \psi \text{ }
\end{equation*}%
for all large dimension functions $\Phi $ and all $\omega \in \Gamma_\psi $.

(iii) Suppose $G^{\prime }(\psi )>\psi $. There is a set $\Gamma_\psi
\subseteq \Omega ,$ of full measure in $\Omega ,$ such that 
\begin{equation*}
\underline{\dim }_{\Phi }\mu _{\omega }\geq \psi \text{ }
\end{equation*}%
for all large dimension functions $\Phi $ and all $\omega \in \Gamma_\psi $.

(iv) Suppose $G^{\prime }(\psi )\leq \psi $. There is a set $\Gamma_\psi
\subseteq \Omega ,$ of full measure in $\Omega ,$ such that%
\begin{equation*}
\underline{\dim }_{\Phi }\mu _{\omega }\leq \psi \text{ }
\end{equation*}%
for all large dimension functions $\Phi $ and all $\omega \in \Gamma_\psi $.
\end{theorem}

An immediate corollary is as follows. Again, there is a corresponding
statement for $G^{\prime }$ and the lower $\Phi $-dimensions.

\begin{corollary}
\label{cor:large_dim} \label{CDim}Suppose there is a choice of $\alpha $
such that $G(\alpha )=\alpha $ and $G(\psi )<\psi $ if $\psi >\alpha $. Then
there is a set $\Gamma \subseteq \Omega ,$ of full measure in $\Omega ,$
such that 
\begin{equation*}
\overline{\dim }_{\Phi }\mu _{\omega }=\alpha \text{ }
\end{equation*}%
for all large dimension functions $\Phi $ and all $\omega \in \Gamma $.
\end{corollary}

\begin{proof}
From part (i) of the theorem, for each rational $q > \alpha$ we have a set $%
\Gamma_q$ of full measure so that for all large dimension functions $\Phi$
and $\omega \in \Gamma_q$ we have $\overline{\dim}_\Phi \mu_\omega \le q$.
From part (ii) of the theorem there is a set $\Gamma_\alpha$ of full measure
so that for all large dimension functions $\Phi$ and $\omega \in
\Gamma_\alpha$ we have $\overline{\dim}_\Phi \mu_\omega \ge \alpha$. Let 
\begin{equation*}
\Gamma = \Gamma_\alpha \cap \bigcap_{q > \alpha, q \text{ rational }}
\Gamma_q,
\end{equation*}
which is also a subset of $\Omega$ of full measure. Then for any large
dimension function $\Phi$ and $\omega \in \Gamma$, we have 
\begin{equation*}
\alpha \le \overline{\dim}_\Phi \mu_\omega \le \inf \{ q : q > \alpha, q 
\text{ rational } \} = \alpha.
\end{equation*}
\end{proof}

Of course, it is enough that $G(\psi _{k})<\psi _{k}$ for a sequence $(\psi
_{k})$ decreasing to $\alpha $.

\begin{corollary}
\label{cor:Gfuncroot} Let $\Phi$ be a large dimension function. Then $\alpha
= \overline{\dim}_\Phi \mu_\omega$ almost surely if and only if $G(\psi) <
\psi$ for all $\psi > \alpha$ and $G(\psi) \ge \psi$ for all $\psi < \alpha$.
\end{corollary}

\begin{proof}
Suppose that $\alpha \ge 0$ is the almost sure value for $\overline{\dim }%
_{\Phi }\mu _{\omega }$ (which we know exists by Proposition \ref%
{prop:almostsuredim}). Take $\psi >\alpha $ and suppose that $G(\psi )\geq
\psi $. Then by part (ii) of the theorem, $\overline{\dim }_{\Phi }\mu
_{\omega }\geq \psi >\alpha $ almost surely, which is a contradiction. Thus
in fact $G(\psi )<\psi $. Similarly, if $\psi <\alpha $ but $G(\psi )<\psi ,$
then $\overline{\dim }_{\Phi }\mu _{\omega }\leq \psi <\alpha $ almost
surely, which is another contradiction and so $G(\psi )\geq \psi $ in this
case.

For the converse, suppose $G(\psi) < \psi$ for all $\psi > \alpha$ and $%
G(\psi) \ge \psi$ for all $\psi < \alpha$. Then for all $\psi > \alpha$ we
have $\overline{\dim}_\Phi \mu_\omega \le \psi$ almost surely and so $%
\overline{\dim}_\Phi \mu_\omega \le \alpha$ almost surely. Similarly for all 
$\psi < \alpha$ we have $\overline{\dim}_\Phi \mu_\omega \ge \psi$ almost
surely and so $\overline{\dim}_\Phi \mu_\omega \ge \alpha$ almost surely.
\end{proof}

What this last corollary shows, in particular, is that there must always be
such a value $\alpha $ where $G$ \textquotedblleft crosses the
diagonal\textquotedblright\ since for any given large $\Phi $ it is clear
that $\overline{\dim }_{\Phi }\mu _{\omega }$ must have some almost sure
value.

\smallskip

Before proving the theorem, we introduce further notation and establish some
preliminary results. Given a large dimension function $\Phi ,$ assume $H$
and $t_{0}$ satisfy%
\begin{equation}
\Phi (t)\geq \frac{H(t)\log \left\vert \log t\right\vert }{\left\vert \log
t\right\vert }\text{ for all }0<t\leq t_{0},  \label{Hdef}
\end{equation}%
where $H(t)$ $\uparrow \infty $ as $t\rightarrow 0$. Set 
\begin{equation}
\zeta _{N}^{H}=\frac{H(B^{N})\log (N\left\vert \log B\right\vert )}{%
\left\vert \log A\right\vert }.  \label{PsiDefn}
\end{equation}

\begin{lemma}
\label{prelim}(i) If $k<\zeta _{N}^{H}$, then for $N$ sufficently large
there are no pairs of Moran subsets $I_{N}(x),$ $I_{N+k}(x)$ where 
\begin{equation*}
|I_{N+k}(x)|\leq |I_{N}(x)|^{1+\Phi (|I_{N}(x)|)}.
\end{equation*}

(ii) Fix $c>0$. If $H$ is sufficiently large near 0, then $%
\sum_{N=1}^{\infty }\exp (-c\zeta _{N}^{H})<\infty $.
\end{lemma}

\begin{proof}
(i) Choose $N_{0}$ such that $B^{N_{0}}\leq t_{0}$. Assume $N\geq N_{0}$ and
for convenience put $r=|I_{N+k}(x)|$ and $R=|I_{N}(x)|\leq B^{N}\leq t_{0}$.
Then 
\begin{eqnarray*}
\Phi (R)|\log R| &\geq &H(R)\log |\log R|\geq H(B^{N})\log |N\log B| \\
&=&\zeta _{N}^{H}|\log A|>k|\log A|,
\end{eqnarray*}%
so $R^{\Phi (R)}<A^{k}$. As $r/R\geq A^{k}>R^{\Phi (R)},$ this means $%
r>R^{1+\Phi (R)}$, hence we cannot have $|I_{N+k}(x)|\leq |I_{N}(x)|^{1+\Phi
(|I_{N}(x)|)}.$

(ii) A straightforward calculation shows that if $H(B^{N})$ is suitably
large for $N\geq N_{0}$, then $\exp (-c\zeta _{N}^{H})\leq N^{-2}$ and hence 
$\sum_{N\geq N_{0}}\exp (-c\zeta _{N}^{H})<\infty .$
\end{proof}

The next lemma is the key probabilistic result. It is based on the Chernov
bounds for iid random variables $X_{1},X_{2},\ldots ,X_{n}$.

\begin{theorem}
{\rm \lbrack Chernov]} \cite[Section 2.2]{Pe} Let $X_{n}$ be iid random variables
and assume that $\mathbb{E}(\exp (t(X_{i}-\mathbb{E}(X_{i}))))\leq
e^{ct^{2}/2}$ for some $c>0$ and $t>0$. Then for all $0<\eta \leq nct$ we
have 
\begin{equation*}
\mathcal{P}\left( \left\vert \sum_{i=1}^{n}X_{i}-n\mathbb{E}%
(X_{i})\right\vert \geq \eta \right) \leq 2e^{\eta ^{2}/(2nc)}.
\end{equation*}
\end{theorem}

\begin{lemma}
\label{PR}{\rm [Probabilistic Result]} Fix any $\theta \geq 0$ and $\delta >0$. If
the constant function $H$ is sufficiently large and $\zeta _{N}^{H}$ is
defined as in (\ref{PsiDefn}), then 
\begin{equation*}
\mathcal{P}\left\{ \omega :\exists m\geq \zeta _{N}^{H}\text{ with }%
\left\vert \frac{\sum_{n=N+1}^{N+m}Y_{n}(\theta )(\omega )}{%
\sum_{n=N+1}^{N+m}Z_{n}(\theta )(\omega )}-\frac{\mathbb{E}(Y_{1}(\theta ))}{%
\mathbb{E}(Z_{1}(\theta ))}\right\vert >\delta \ \mathrm{\ i.o.}\right\} =0%
\text{.}
\end{equation*}%
A similar statement holds for $Y_{n}^{\prime },$ $Z_{n}^{\prime }$.
\end{lemma}

\begin{proof}
Since the function $f(x,y)=x/y$ is continuous at the point $(\mathbb{E}%
(Y_{1}(\theta )),\mathbb{E}(Z_{1}(\theta ))),$ for the given $\delta >0$
there is some $\eta =\eta (\delta )>0$ such that when both inequalities 
\begin{equation*}
\left\vert \frac{1}{m}\sum_{n=N+1}^{N+m}Y_{n}(\theta )(\omega )-\mathbb{E}%
(Y_{1}(\theta ))\right\vert \leq \eta \text{ and }\left\vert \frac{1}{m}%
\sum_{n=N+1}^{N+m}Z_{n}(\theta )(\omega )-\mathbb{E}(Z_{1}(\theta
))\right\vert \leq \eta
\end{equation*}%
hold, then 
\begin{equation*}
\left\vert \frac{\sum_{n=N+1}^{N+m}Y_{n}}{\sum_{n=N+1}^{N+m}Z_{n}}-\frac{%
\mathbb{E}(Y_{1})}{\mathbb{E}(Z_{1})}\right\vert \leq \delta .
\end{equation*}

Since we have assumed $\mathbb{E}(e^{-t\log p_{n}}),\mathbb{E}(e^{-t\log
(1-p_{n})})<\infty $, Chernov's inequality implies there are constants $C$
and $c>0$ such that for all $m,$ 
\begin{equation*}
\mathcal{P}\left\{ \omega :\left\vert \frac{1}{m}\sum_{n=N+1}^{N+m}Y_{n}-%
\mathbb{E}(Y_{1})\right\vert >\eta \right\} \leq Ce^{-cm}.\text{ }
\end{equation*}%
Applying Lemma \ref{prelim}(ii), we know $\sum_{N}e^{-c\zeta
_{N}^{H}}<\infty $ if $H$ is sufficiently large. Thus, if we let%
\begin{equation*}
\Gamma _{N,\eta }=\{\omega :\exists m\geq \zeta _{N}^{H}\text{ with }%
\left\vert \frac{1}{m}\sum_{n=N+1}^{N+m}Y_{n}-\mathbb{E}(Y_{1})\right\vert
>\eta \},
\end{equation*}%
then for a new constant $C_{1},$%
\begin{equation*}
\sum_{N=1}^{\infty }\mathcal{P(}\Gamma _{N,\eta })\leq \sum_{N}\sum_{m=\zeta
_{N}^{H}}^{\infty }Ce^{-cm}\leq \sum_{N}C_{1}e^{-c\zeta _{N}^{H}}<\infty .
\end{equation*}%
By the Borel-Cantelli lemma this means $\mathcal{P(}\Gamma _{N,\eta }$ i.o.$%
)=0$.

Similarly, if we let $\Gamma _{N,\eta }^{\prime }=\{\omega :\exists m\geq
\zeta _{N}^{H}$ with $\left\vert \sum_{n=N+1}^{N+m}Z_{n}-\mathbb{E}%
(Z_{1})\right\vert >\eta \},$ then for a suitable choice of $H$ we have $%
\mathcal{P(}\Gamma _{N,\eta }^{\prime }$ i.o.$)=0$.

Hence there is a set $\Omega (\eta ),$ of full measure, with the property
that for each $\omega \in \Omega (\eta )$ there is some $N_{\eta }=N_{\eta
}(\omega )$ such that for all $N\geq N_{\eta }$ and all $m\geq \zeta
_{N}^{H},$ we have both 
\begin{equation*}
\left\vert \frac{1}{m}\sum_{n=N+1}^{N+m}Y_{n}-\mathbb{E}(Y_{1})\right\vert
\leq \eta \text{ and }\left\vert \frac{1}{m}\sum_{n=N+1}^{N+m}Z_{n}-\mathbb{E%
}(Z_{1})\right\vert \leq \eta \text{. }
\end{equation*}%
and therefore, 
\begin{equation*}
\left\vert \frac{\sum_{n=N+1}^{N+m}Y_{n}}{\sum_{n=N+1}^{N+m}Z_{n}}-\frac{%
\mathbb{E}(Y_{1})}{\mathbb{E}(Z_{1})}\right\vert \leq \delta .
\end{equation*}%
That completes the proof.
\end{proof}

\subsection{Proof of the Theorem}

\begin{proof}
\lbrack Proof of Theorem \ref{MainLarge}] (i) For each positive integer $j,$
let 
\begin{equation*}
\Phi _{j}(t)=\frac{j\log \left\vert \log t\right\vert }{\left\vert \log
t\right\vert }\text{ and }\zeta _{N}^{j}=\frac{j\log (N\left\vert \log
B\right\vert )}{\left\vert \log A\right\vert }.
\end{equation*}

Consider any $N,m\in \mathbb{N}$, $\psi >0$, Moran interval $I_{N}(\omega )$
and descendent interval $I_{N+m}(\omega )$. If $I_{N}=I_{v}$ for $%
v=v_{1}\cdot \cdot \cdot v_{N}$ with $v_{i}\in \{0,1\}$ and $%
I_{N+m}=I_{vv_{N+1}\cdot \cdot \cdot v_{N+m}}$, then 
\begin{equation*}
\frac{\mu _{\omega }(I_{N})}{\mu _{\omega }(I_{N+m})}=\left( \prod 
_{\substack{ v_{N+i}=0,  \\ i=1,...,m}}p_{N+i}(\omega )\cdot \prod 
_{\substack{ v_{N+i}=1,  \\ i=1,...,m}}(1-p_{N+i}(\omega ))\right) ^{-1}
\end{equation*}%
and%
\begin{equation*}
\frac{\left\vert I_{N}\right\vert }{\left\vert I_{N+m}\right\vert }=\left(
\prod_{\substack{ v_{N+i}=0,  \\ i=1,...,m}}a_{N+i}(\omega )\cdot \prod 
_{\substack{ v_{N+i}=1,  \\ i=1,...,m}}b_{N+i}(\omega )\right) ^{-1}.
\end{equation*}%
Thus, for any $\psi $, 
\begin{eqnarray}
\frac{\frac{\mu _{\omega }(I_{N})}{\mu _{\omega }(I_{N+m})}}{\left( \frac{%
\left\vert I_{N}\right\vert }{\left\vert I_{N+m}\right\vert }\right) ^{\psi }%
} &=&\left( \prod_{\substack{ v_{N+i}=0,  \\ i=1,...,m}}\frac{a_{N+i}^{\psi }%
}{p_{N+i}}\right) \left( \prod_{\substack{ v_{N+i}=1,  \\ i=1,...,m}}\frac{%
b_{N+i}^{\psi }}{1-p_{N+i}}\right)  \notag \\
&\leq &\prod_{i=N+1}^{N+m}\max \left( \frac{a_{i}^{\psi }(\omega )}{%
p_{i}(\omega )},\frac{b_{i}^{\psi }(\omega )}{1-p_{i}(\omega )}\right) .
\label{eq:YandZ}
\end{eqnarray}%
Now 
\begin{equation*}
\frac{a_{i}^{\psi }}{p_{i}}\geq \frac{b_{i}^{\psi }}{1-p_{i}}\text{ if and
only if }p_{i}\leq \frac{a_{i}^{\psi }}{a_{i}^{\psi }+b_{i}^{\psi }},
\end{equation*}%
hence 
\begin{equation}
\frac{\mu _{\omega }(I_{N})}{\mu _{\omega }(I_{N+m})}\leq \left( \frac{%
\left\vert I_{N}\right\vert }{\left\vert I_{N+m}\right\vert }\right) ^{\psi }
\label{P1}
\end{equation}%
if 
\begin{equation}
\left( \prod_{\substack{ i=N+1,...,N+m;  \\ p_{i}\leq a_{i}^{\psi
}/(a_{i}^{\psi }+b_{i}^{\psi })}}\frac{a_{i}^{\psi }}{p_{i}}\right) \left(
\prod_{\substack{ i=N+1,...,N+m;  \\ p_{i}>a_{i}^{\psi }/(a_{i}^{\psi
}+b_{i}^{\psi })}}\frac{b_{i}^{\psi }}{(1-p_{i})}\right) \leq 1.  \label{P4}
\end{equation}%
Taking logarithms, we see that (\ref{P4}) is equivalent to the statement%
\begin{equation}
\psi \geq \frac{\sum_{i=N+1}^{N+m}Y_{i}(\psi )(\omega )}{%
\sum_{i=N+1}^{N+m}Z_{i}(\psi )(\omega )}.  \label{P2}
\end{equation}

Finally, assume $G(\psi )<\psi $, say $G(\psi )\leq \psi -2\delta $ for some 
$\delta >0$. According to the probabilistic result, Lemma \ref{PR}, there is
a set $\Omega _{j,\psi }$, depending on both $j$ and $\psi $ and of full
measure in $\Omega $, such that for each $\omega \in \Omega _{j,\psi }$
there is some integer $N_{j}=N_{j}(\omega )$ such that for all $N\geq N_{j}$
and all $m\geq \zeta _{N}^{j},$%
\begin{equation*}
\left\vert \frac{\sum_{n=N+1}^{N+m}Y_{n}(\psi )(\omega )}{%
\sum_{n=N+1}^{N+m}Z_{n}(\psi )(\omega )}-\frac{\mathbb{E}(Y_{1}(\psi ))}{%
\mathbb{E}(Z_{1}(\psi ))}\right\vert =\left\vert \frac{%
\sum_{n=N+1}^{N+m}Y_{n}(\psi )(\omega )}{\sum_{n=N+1}^{N+m}Z_{n}(\psi
)(\omega )}-G(\psi )\right\vert \leq \delta .
\end{equation*}%
Consequently,%
\begin{equation*}
\frac{\sum_{n=N+1}^{N+m}Y_{n}(\psi )(\omega )}{\sum_{n=N+1}^{N+m}Z_{n}(%
\psi)(\omega )}\leq G(\psi )+\delta \leq \psi -\delta <\psi .
\end{equation*}%
Thus our previous observations imply that for each $\omega \in \Omega
_{j,\psi }$ there is an integer $N_{j}$ such that for all $N\geq N_{j}$ and
all $m\geq \zeta _{N}^{j},$%
\begin{equation}
\frac{\mu _{\omega }(I_{N})}{\mu _{\omega }(I_{N+m})}\leq \left( \frac{%
\left\vert I_{N}\right\vert }{\left\vert I_{N+m}\right\vert }\right) ^{\psi
}.  \label{P3}
\end{equation}

Next, suppose $\omega \in \Omega _{j,\psi },N_{j}=N_{j}(\omega )$ is as
above and $x\in \mathcal{C}_{\omega }$. Choose $N\geq N_{j}$ and $m$ so that 
$|I_{N}(x)|<A^{N_{j}+L}$ and 
\begin{equation*}
|I_{N+m}(x)|\leq |I_{N}(x)|^{1+\Phi (|I_{N}(x)|)}.
\end{equation*}%
Then Lemma \ref{prelim}(i) implies $m\geq \zeta _{N}^{j}$ and so by (\ref{P3}%
) and Lemma \ref{lem:intervals_give_dimension} we know that $\overline{\dim }%
_{\Phi }\mu _{\omega }\leq \psi $ for all $\omega \in \Omega _{j,\psi }$.

Now, let $\Phi $ be any large dimension function and 
\begin{equation*}
\omega \in \Gamma _{\psi }=\bigcap_{j=1}^{\infty }\Omega _{j,\psi },
\end{equation*}%
again a set of full measure. There exists $j$ such that $\Phi \left(
t\right) \geq \Phi _{j}(t)$ for $t$ sufficiently close to $0$. As $\omega
\in \Omega _{j,\psi },$ $\overline{\dim }_{\Phi }\mu _{\omega }\leq 
\overline{\dim }_{\Phi _{j}}\mu _{\omega }\leq \psi $. It follows that $%
\overline{\dim }_{\Phi }\mu _{\omega }\leq \psi $ for all $\omega \in \Gamma
_{\psi }$ and all large dimension functions $\Phi $.

\smallskip

(ii) Given $\omega ,$ consider the Moran intervals which arise by choosing
the left child at step $n$ if 
\begin{equation*}
\frac{a_{n}^{\psi }(\omega )}{p_{n}(\omega )}=\max \left( \frac{%
a_{n}^{\psi}(\omega )}{p_{n}(\omega )},\frac{b_{n}^{\psi }(\omega )}{%
1-p_{n}(\omega )}\right)
\end{equation*}%
and the right child otherwise. Call the interval at step $n$ which arises by
this construction $I_{n}=I_{n}(\psi ,\omega )$. These form a nested sequence
of Moran intervals.

For $n>N,$ 
\begin{equation*}
\frac{\mu _{\omega }(I_{N})}{\mu _{\omega }(I_{n})}= \prod_{\substack{ i=N+1 
\\ p_{i}\leq a_{i}^{\psi }/(a_{i}^{\psi }+b_{i}^{\psi })}}^{n}p_{i}^{-1}
\prod _{\substack{ i=N+1  \\ p_{i}>a_{i}^{\psi }/(a_{i}^{\psi}+b_{i}^{\psi
}) }}^{n}(1-p_{i})^{-1}
\end{equation*}%
and 
\begin{equation*}
\frac{\left\vert I_{N}\right\vert }{\left\vert I_{n}\right\vert }%
=\prod_{p_{i}\leq a_{i}^{\psi }/(a_{i}^{\psi }+b_{i}^{\psi})}a_{i}^{-1}
\prod_{p_{i}>a_{i}^{\psi }/(a_{i}^{\psi }+b_{i}^{\psi})}b_{i}{}^{-1}.
\end{equation*}%
Thus, for any $\beta >0,$ 
\begin{equation*}
\frac{\mu _{\omega }(I)}{\mu _{\omega }(I_{n})}\geq \left( \frac{\left\vert
I_{N}\right\vert }{\left\vert I_{n}\right\vert }\right) ^{\beta }
\end{equation*}%
if and only if 
\begin{equation*}
\sum_{i=N+1}^{n}Y_{i}(\psi )(\omega )\leq \beta
\sum_{i=N+1}^{n}Z_{i}(\psi)(\omega )
\end{equation*}%
if and only if (writing $n=N+m)$ 
\begin{equation*}
\frac{\sum_{i=N+1}^{N+m}Y_{i}(\psi )(\omega )}{\sum_{i=N+1}^{N+m}Z_{i}(%
\psi)(\omega )}\geq \beta .
\end{equation*}

Fix $\delta >0$ and choose the constant function $H=H(\delta )$ so large
that Lemma \ref{PR} guarantees that there is a set $\Omega _{\delta,\psi},$
of full measure, such that for all $\omega \in \Omega _{\delta,\psi}$ and $N$
sufficiently large, 
\begin{equation*}
\left\vert \frac{\sum_{i=N+1}^{N+m}Y_{i}(\psi)(\omega )}{%
\sum_{i=N+1}^{N+m}Z_{i}(\psi)(\omega )}-G(\psi )\right\vert \leq \delta 
\text{ for all }m\geq \zeta _{N}^{H},
\end{equation*}%
and hence%
\begin{equation*}
\frac{\sum_{i=N+1}^{N+m}Y_{i}}{\sum_{i=N+1}^{N+m}Z_{i}}\geq G(\psi )-\delta
\geq \psi -\delta \text{.}
\end{equation*}%
It follows that%
\begin{equation*}
\frac{\mu _{\omega }(I_{N})}{\mu _{\omega }(I_{N+m})}\geq \left( \frac{%
\left\vert I_{N}\right\vert }{\left\vert I_{N+m}\right\vert }\right) ^{\psi
-\delta }
\end{equation*}%
for all $\omega \in \Omega _{\delta,\psi}$, $m\geq \zeta _{N}^{H}$ and $N$
sufficiently large.

Now, take $\delta _{j}=1/j$, let $H(\delta _{j})=H_{j}$ and $\Omega_j=\Omega
_{\delta_j,\psi}$. Let $\Gamma_\psi =\bigcap_{j}\Omega _{j}$, a set of full
measure. As $\Phi $ is a large dimension function, for any $j$ there exists $%
t_{j}>0$ such that $\Phi (t)=H(t)\log \left\vert \log t\right\vert
/\left\vert \log t\right\vert $ where $H(t)\geq H_{j}$ for $t\leq t_{j}$.
Consequently, for large $N,$ $\zeta _{N}^{H}\geq \zeta _{N}^{H_{j}}$. If $%
\omega \in \Gamma_\psi$, then $\omega \in \Omega _{j}$ and therefore 
\begin{equation*}
\frac{\mu _{\omega }(I_{N})}{\mu _{\omega }(I_{N+\zeta _{N}^{H}})}\geq
\left( \frac{\left\vert I_{N}\right\vert }{\left\vert I_{N+\zeta
_{N}^{H}}\right\vert }\right) ^{\psi -1/j}
\end{equation*}%
for all $N$ sufficiently large. It follows that for all $\omega \in
\Gamma_\psi ,$ $\overline{\dim }_{\Phi }\mu _{\omega }\geq \psi -1/j$ and
since this is true for all $j$, we must have $\overline{\dim }_{\Phi }\mu
_{\omega }\geq \psi $ as claimed.

The arguments for the lower $\Phi $-dimension are very similar, but rather
than considering $\max \left( \frac{a_{n}^{\psi }}{p_{n}},\frac{b_{n}^{\psi }%
}{1-p_{n}}\right) ,$ we study $\min \left( \frac{a_{n}^{\psi }}{p_{n}},\frac{%
b_{n}^{\psi }}{1-p_{n}}\right) $. Thus the functions $Y_{n}^{\prime},Z_{n}^{%
\prime }$ and $G^{\prime }$ arise in place of $Y_{n},Z_{n}$ and $G$. The
details are left for the reader. \newline
\end{proof}

\subsection{Consequences of the Theorem}

We continue to use the notation introduced earlier. In particular, $G$ is as
defined in (\ref{G}). Since positive constant functions are large dimension
functions, the following corollary follows directly from the theorem.

\begin{corollary}
\label{cor:large_quasi} (i) If $G(\psi )<\psi $, then $\dim _{qA}\mu
_{\omega }\leq \psi $ a.s.

(ii) If $G(\psi )\geq \psi $, then $\dim _{qA}\mu _{\omega }\geq \psi $ a.s.
\end{corollary}

Similar statements hold for $G^{\prime }$ and the quasi-lower Assouad
dimension.

A useful fact, which we show below, is that continuous functions $G$ (or $%
G^{\prime }$) typically satisfy the hypotheses of Corollaries \ref%
{cor:large_dim} and \ref{cor:Gfuncroot}. This is often the situation, c.f. (%
\ref{GFormula}) where it is shown that $G$ is even differentiable when $%
a_{n}=a,$ $b_{n}=b$ and $p_{n} $ is uniformly distributed over $[0,1]$. More
generally, $G$ is continuous if $p_{n}$ has a density distribution of the
form $f(t)dt$, where $f(t)\log t $ and $f(t)\log (1-t)$ are integrable over $%
[0,1]$, such as when $f$ is bounded.

\begin{lemma}
\label{alpha} Assume $\left\vert \mathbb{E}(\log p_{1})\right\vert $, $%
\left\vert \mathbb{E}(\log (1-p_{1}))\right\vert <\infty $. If $G(\theta )$
is continuous, then there is a unique choice of $\alpha $ such that $%
G(\alpha )=\alpha $ and $G(\psi )<\psi $ if $\psi >\alpha $.
\end{lemma}

\begin{proof}
We will assume that $\mathcal{P}(a_{n}=b_{n})=0$ and leave the contrary case
to the reader. Note that as $\theta \rightarrow \infty ,$ 
\begin{equation*}
\frac{a^{\theta }}{a^{\theta }+b^{\theta }}\rightarrow \left\{ 
\begin{array}{cc}
0 & \text{if }a<b \\ 
1 & \text{if }a>b%
\end{array}%
\right. ,
\end{equation*}%
and therefore 
\begin{equation*}
Y_{1}(\theta )(\omega )\rightarrow \left\{ 
\begin{array}{cc}
\log (1-p_{1}(\omega )) & \text{if }a_{1}(\omega )<b_{1}(\omega ) \\ 
\log p_{1}(\omega ) & \text{if }a_{1}(\omega )>b_{1}(\omega )%
\end{array}%
\right. \quad \text{ as }\theta \rightarrow \infty
\end{equation*}%
and 
\begin{equation*}
Z_{1}(\theta )(\omega )\rightarrow \left\{ 
\begin{array}{cc}
\log b_{1}(\omega ) & \text{if }a_{1}(\omega )<b_{1}(\omega ) \\ 
\log a_{1}(\omega ) & \text{if }a_{1}(\omega )>b_{1}(\omega )%
\end{array}%
\right. \quad \text{ as }\theta \rightarrow \infty .
\end{equation*}%
Hence 
\begin{equation*}
G(\theta )\rightarrow \frac{\mathcal{P}(a_{1}<b_{1})\mathbb{E}(\log
(1-p_{1}))+\mathcal{P}(a_{1}>b_{1})\mathbb{E}(\log p_{1})}{\mathcal{P}%
(a_{1}<b_{1})\mathbb{E}(\log b_{1})+\mathcal{P}(a_{1}>b_{1})\mathbb{E}(\log
a_{1})}\text{ as }\theta \rightarrow \infty.
\end{equation*}%
In particular, $G$ approaches a (finite) constant as $\theta \rightarrow
\infty $.

On the other hand, 
\begin{equation*}
G(0 ) = \frac{\mathbb{E}(\log p_{1}|_{p_{1}\leq 1/2})+\mathbb{E}(\log
(1-p_{1})|_{p_{1}>1/2})}{\mathcal{P}(p_{1}\leq 1/2)\mathbb{E}(\log a_{1})+%
\mathcal{P}(p_{1}>1/2)\mathbb{E}(\log b_{1})}\text{ }>0.
\end{equation*}

Since $G$ is continuous, $G(0)>0$ and eventually $G(\theta )<\theta $, there
must be a unique choice of $\alpha $ such that $G(\alpha )=\alpha $ and if $%
\psi >\alpha ,$ then $G(\psi )<\psi $.
\end{proof}

\begin{corollary}
\label{Cor:Cont} Suppose $\left\vert \mathbb{E}(\log p_{1})\right\vert $, $%
\left\vert \mathbb{E}(\log (1-p_{1}))\right\vert <\infty $ and $G(\theta )$
is continuous$.$ Then $\overline{\dim }_{\Phi }\mu _{\omega }=\alpha $ a.s.
where $G(\alpha )=\alpha $ and $G(\psi )<\psi $ for all $\psi >\alpha $.
\end{corollary}

The upper $\Phi $-dimension of $\mu $ is always an upper bound for the upper
local dimension of $\mu $ at any point $x$, where the latter is defined by 
\begin{equation}  \label{eq:upperlocaldim}
\overline{\dim }_{\text{loc}}\mu (x)=\limsup_{r\rightarrow 0}\frac{\log \mu
(B(x,r))}{\log r}.
\end{equation}

In a similar way the lower $\Phi$-dimension of $\mu$ is always a lower bound
for the lower local dimension (defined as in (\ref{eq:upperlocaldim}) but
using a liminf). However, in general it is possible for $\underline{\dim}%
_\Phi \mu < \inf_x \underline{\dim}_{\text{loc}}\mu(x)$ and $\sup_{x}%
\overline{\dim }_{\text{loc}}\mu (x)<\overline{\dim }_{\Phi }\mu $. (See 
\cite{HH} for proofs of these statements.) In the case of our random
measures, there is no gap for either inequality.

\begin{proposition}
Assume $G(\psi )<\psi $ for all $\psi >\theta $ and $G(\theta )=\theta $.
Then for any large dimension function $\Phi $ and almost all $\omega $ we
have that 
\begin{equation*}
\sup_{x}\overline{\dim }_{\rm{loc}}\mu _{\omega }(x)=\theta =\overline{%
\dim }_{\Phi }\mu _{\omega }.
\end{equation*}%
Similarly, if $G^{\prime }(\psi )>\psi $ for all $\psi <\theta ^{\prime }$
and $G^{\prime }(\theta ^{\prime })=\theta ^{\prime }$ then for any large
dimension function $\Phi $ and almost all $\omega $ we have that 
\begin{equation*}
\inf_{x}\underline{\dim }_{\rm{loc}}\mu _{\omega }(x)=\theta ^{\prime }=%
\underline{\dim }_{\Phi }\mu _{\omega }.
\end{equation*}
\end{proposition}

\begin{proof}
Put $v_{j}=0$ if $a_{j}^{\theta }/p_{j}=\max (a_{j}^{\theta
}/p_{j},b_{j}^{\theta }/(1-p_{j}))$ and $v_{j}=1$ else. Let $x\in
\bigcap_{n=1}^{\infty }I_{v_{1},....,v_{n}},$ so that $%
I_{n}(x)=I_{v_{1},....,v_{n}}$ for each $n$. Given any small $r>0,$ choose $%
n $ such that $\left\vert I_{n}(x)\right\vert \leq r<\left\vert
I_{n-1}(x)\right\vert $, so that $I_{n}(x)\subseteq B(x,r)\leq I_{n-L}(x)$.

We have 
\begin{equation*}
\mu_\omega (B(x,r))\leq \mu_\omega
(I_{n-L}(x))=\prod_{j=1;v_{j}=0}^{n-L}p_{j}%
\prod_{j=1;v_{j}=1}^{n-L}(1-p_{j}),
\end{equation*}%
so 
\begin{equation*}
\left\vert \log \mu_\omega (B(x,r))\right\vert \geq \left\vert
\sum_{i=1}^{n-L} Y_{i}(\theta )\right\vert .
\end{equation*}%
Similarly, 
\begin{equation*}
r\geq \prod_{j=1;v_{j}=0}^{n}a_{j}\prod_{j=1;v_{j}=1}^{n}b_{j}=\left(
\prod_{j=1;v_{j}=0}^{n-L}a_{j}\prod_{j=1;v_{j}=1}^{n-L}b_{j}\right) \left(
\prod_{\substack{ j=n-L+1;  \\ v_{j}=0}}^{n}a_{j}\prod_{\substack{ j=n-L+1; 
\\ v_{j}=1}}^{n}b_{j}\right) .
\end{equation*}%
As $a_{j},b_{j}$ are bounded away from $0$ and $L$ is fixed, there is some
constant $C>0$ such that%
\begin{equation*}
\left\vert \log r\right\vert \leq |\sum_{i=1}^{n-L} Z_i(\theta )|+C.
\end{equation*}

Hence%
\begin{equation*}
\frac{\left\vert \log \mu_\omega (B(x,r))\right\vert }{\left\vert \log
r\right\vert } \geq \frac{\left\vert \displaystyle \sum_{i=1}^{n-L}
Y_i(\theta )\right\vert }{\left|\displaystyle \sum_{i=1}^{n-L} Z_{i}(\theta
) \right|+C}.
\end{equation*}

Fix $\varepsilon >0$ and choose a set of full measure, $\Omega_{%
\varepsilon}, $ such that 
\begin{equation*}
\frac{ \displaystyle \sum_{i=1}^m Y_{i}(\theta )(\omega)}{ \displaystyle %
\sum_{i=1}^m Z_{i}(\theta )(\omega)} \geq G(\theta )-\varepsilon
\end{equation*}
for $m\geq m_{\omega }$ and each $\omega \in \Omega_\epsilon$. Further, as $%
\left\vert \sum_{i=1}^m Z_{i}\right\vert \rightarrow \infty $ as $%
m\rightarrow \infty $, given any $\delta >0$ we can choose $m_{0}$
sufficiently large so that for all $m\geq m_{0}$ we have $%
C\left|\sum_{i=1}^m Z_{i}(\theta )\right|^{-1} \leq \delta $. Thus, for all $%
\omega \in \Omega _{\varepsilon }$ and all $n\geq \max (m_{0},m_{\omega }) +
L,$%
\begin{equation*}
\frac{\left\vert \sum_{i=1}^{n-L} Y_{i}(\theta )(\omega)\right\vert }{\left|
\sum_{i=1}^{n-L} Z_{i}(\theta )(\omega) \right|+C}\geq \frac{G(\theta
)-\varepsilon }{1+C\left\vert \sum_{i=1}^{n-L} Z_{i}(\theta
)(\omega)\right\vert^{-1} }\geq \frac{\theta -\varepsilon }{1+\delta }.
\end{equation*}%
If we make the choice of $\delta $ sufficiently small, depending on $\theta
, $ then we can conclude that 
\begin{equation*}
\frac{\left\vert \log \mu _{\omega }(B(x,r))\right\vert }{\left\vert \log
r\right\vert }\geq \theta -2\varepsilon
\end{equation*}%
for sufficiently small $r$. By choosing the sequence $\varepsilon =1/k$ and
putting $\Omega _{0}=\bigcap_{k=1}^{\infty }\Omega _{1/k},$ we deduce that
for all $\omega \in \Omega _{0},$ a set of full measure%
\begin{equation*}
\overline{\dim}_{\text{loc}} \mu_\omega(x) = \limsup_{r\rightarrow 0}\frac{%
\left\vert \log \mu (B(x,r))\right\vert }{\left\vert \log r\right\vert }\geq
\theta .
\end{equation*}

The claim follows since it is always true that the supremum of the upper
local dimensions is dominated by $\overline{\dim }_{\Phi }\mu $ for any
dimension function $\Phi$ (see \cite{HH}) which, according to Corollary \ref%
{cor:large_dim}, is equal to $\theta$ almost everywhere.

The statement about the lower local dimension and $G^{\prime }$ is proved in
an analogous manner.
\end{proof}

\subsection{Example: The deterministic Moran set $\mathcal{C}_{ab}$}

\label{subsec:Cab}

Consider the deterministic Moran set $\mathcal{C}_{ab},$ which can be viewed
as a random Moran set where $(a_{n},b_{n})$ is chosen from the singleton $%
\{(a,b)\}$.

\subsubsection{ $p_n$ chosen uniformly over $[0,1]$}

Suppose $p_n$ has the uniform distribution over $[0,1]$. Then, we have 
\begin{eqnarray*}
\mathbb{E}(Y(\theta )(\omega )) &=&\int_{p_n(\omega)\leq a^{\theta
}/(a^{\theta }+b^{\theta })} \hskip -1.6 cm \log p_n(\omega)\text{ }d\mathbb{%
P(\omega )+}\int_{p_n(\omega)>a^{\theta }/(a^{\theta }+b^{\theta })} \hskip %
-1.6 cm \log (1-p_n(\omega))\text{ }d\mathbb{P(\omega )} \\
&=&\frac{a^{\theta }}{a^{\theta }+b^{\theta }}\log \left( \frac{a^{\theta }}{%
a^{\theta }+b^{\theta }}\right) +\frac{b^{\theta }}{a^{\theta }+b^{\theta }}%
\log \left( \frac{b^{\theta }}{a^{\theta }+b^{\theta }}\right) -1
\end{eqnarray*}%
and%
\begin{eqnarray*}
\mathbb{E}(Z(\theta )(\omega )) &=&(\log a)\mathcal{P}\left( p_n \leq
a^{\theta }/(a^{\theta }+b^{\theta })\right) +(\log b)\mathcal{P}\left( p_n
>a^{\theta }/(a^{\theta }+b^{\theta })\right) \\
&=&\frac{a^{\theta }}{a^{\theta }+b^{\theta }}\log a+\frac{b^{\theta }}{%
a^{\theta }+b^{\theta }}\log b.
\end{eqnarray*}%
Consequently,%
\begin{eqnarray}
G(\theta ) &=&\frac{a^{\theta }\log \left( \frac{a^{\theta }}{a^{\theta
}+b^{\theta }}\right) -a^{\theta }+b^{\theta }\log \left( \frac{b^{\theta }}{%
a^{\theta }+b^{\theta }}\right) -b^{\theta }}{a^{\theta }\log a+b^{\theta
}\log b}.  \label{GFormula}
\end{eqnarray}%
One can clearly see that $G$ is a continuous function (even differentiable)
and so Corollary \ref{Cor:Cont} applies.

Choose $\gamma $ so that $a=b^{\gamma }$. Then $G(\theta )=\theta $ if and
only if 
\begin{equation*}
-(b^{\theta \gamma }+b^{\theta })\log (1+b^{\theta (1-\gamma )})-(b^{\theta
\gamma }+b^{\theta })+\theta b^{\theta }(1-\gamma )\log b=\theta (\log
b)(\gamma b^{\theta \gamma }+b^{\theta })
\end{equation*}%
if and only if%
\begin{equation*}
-(b^{\theta \gamma }+b^{\theta })\log (1+b^{\theta (1-\gamma )})-(b^{\theta
\gamma }+b^{\theta })=\theta \gamma (\log b)(b^{\theta \gamma }+b^{\theta }).
\end{equation*}%
Dividing through by $b^{\theta \gamma }+b^{\theta },$ this is equivalent to
the statement%
\begin{equation*}
\log (1+b^{\theta (1-\gamma )})+1=-\gamma \theta \log b.
\end{equation*}%
Taking the exponential of both sides, it follows that $G(\theta )=\theta $
if and only if 
\begin{equation*}
b^{\theta }+b^{\theta \gamma }-e^{-1}=0\text{.}
\end{equation*}

\begin{example}
Suppose $\mathcal{C}_{ab}$ is the deterministic Moran set with $a=b^{2}$, $%
p_{n}$ is uniformly distributed over $[0,1]$ and $\mu $ is the corresponding
random measure. The analysis above shows that $G(\theta )=\theta $ if and
only if 
\begin{equation*}
b^{2\theta }+b^{\theta }-e^{-1}=0,
\end{equation*}%
equivalently, $b^{\theta }=\left( -1\pm \sqrt{1+4e^{-1}}\right) /2$. Hence
according to Cor. \ref{Cor:Cont}, for all large $\Phi ,$ 
\begin{equation*}
\overline{\dim }_{\Phi }\mu =\frac{\log \left( \frac{\sqrt{1+4e^{-1}}-1}{2}%
\right) }{\log b}\text{ a.s.}
\end{equation*}%
For example, if $b=1/2$ and $a=1/4,$ then $\overline{\dim }_{\Phi }\mu
\approx (1.25)/\log 2$.

It is interesting that the ratio of $\overline{\dim }_{\Phi }\mu $ to $\dim
_{H}\mathcal{C}_{ab}$ is constant (and approximately $2.60$) for these
measures. We see this since 
\begin{equation*}
\dim _{H}\mathcal{C}_{ab}=\frac{\log (\sqrt{5}/2-1/2)}{\log b}
\end{equation*}%
is the non-negative solution to $b^{2d}+b^{d}=1$. (We note that because of
self-similarity and the separation condition, all the \textquotedblleft
usual\textquotedblright\ dimensions of $\mathcal{C}_{ab}$ agree with the
similarity dimension.)
\end{example}

\begin{example}
We continue with $\mathcal{C}_{ab}$ as the deterministic Moran set with $%
p_{n}$ being drawn uniformly from $[0,1]$ and with $\mu $ as the
corresponding random measure. In Figure \ref{fig:Cab_unifP} (obtained by
numerically solving $G(\theta )=\theta $) we show the almost sure upper $%
\Phi $-dimension (for large $\Phi $) of $\mu $ on $\mathcal{C}_{ab}$ as a
function of $(a,b)$ where, for this figure, we draw $(a,b)$ from the set 
\begin{equation*}
\Lambda =\{(a,b):1/50\leq \min \{a,b\}\leq a+b\leq 49/50\}.
\end{equation*}%
\begin{figure}[tph]
\begin{center}
\includegraphics[width = 3 in]{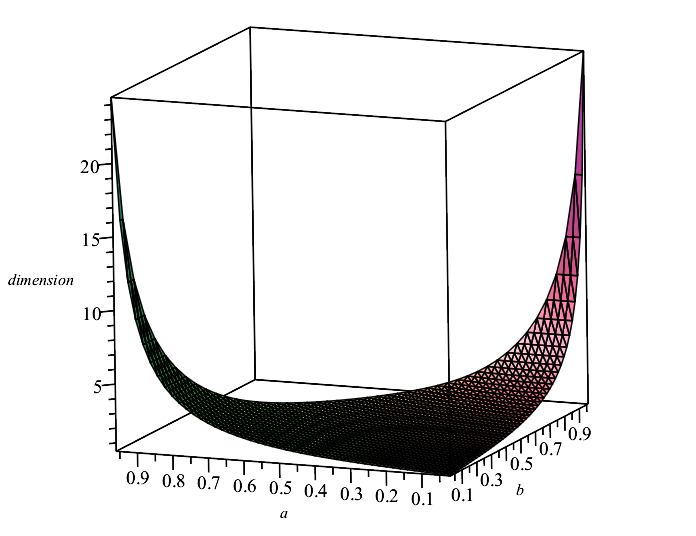}
\end{center}
\caption{$\overline{\dim }_{\Phi }\protect\mu $ as a function of $(a,b)$ for 
$\protect\mu $ on $C_{ab}$ with $p\sim U[0,1]$.}
\label{fig:Cab_unifP}
\end{figure}
It is notable that the dimension is a continuous function of $(a,b)\in
\Lambda $ and it appears to increase as either $a\rightarrow 1$ or $%
b\rightarrow 1$ and vanish as $a$ and $b$ both tend to $0$. In fact that is
indeed the case as we now argue.

For our discussion, let $D_{a b}$ be the almost sure upper $\Phi $-dimension
of the random measure $\mu $ for the large $\Phi$ case. What we wish to show
is that $D_{a b}\rightarrow 0$ as $a$ and $b$ tend to $0$ and $D_{a
b}\rightarrow \infty $ as $a$ or $b$ tend to $1$.

\begin{proof}
From (\ref{GFormula}) we have%
\begin{equation*}
G(\theta )=\frac{a^{\theta }\log \left( \frac{a^{\theta }}{a^{\theta
}+b^{\theta }}\right) -a^{\theta }+b^{\theta }\log \left( \frac{b^{\theta }}{%
a^{\theta }+b^{\theta }}\right) -b^{\theta }}{a^{\theta }\log a+b^{\theta
}\log b}.
\end{equation*}%
and thus $G(\theta )\geq \theta $ if and only if 
\begin{equation}
a^{\theta }\log \left( \frac{a^{\theta }}{a^{\theta }+b^{\theta }}\right)
-a^{\theta }+b^{\theta }\log \left( \frac{b^{\theta }}{a^{\theta }+b^{\theta
}}\right) -b^{\theta }\leq \theta (a^{\theta }\log a+b^{\theta }\log b).
\label{lower}
\end{equation}%
After some simplification, we see that this happens precisely when 
\begin{equation*}
(a^{\theta }+b^{\theta })(1+\log (a^{\theta }+b^{\theta }))\geq 0
\end{equation*}%
which, since $a^{\theta }+b^{\theta }>0$, is equivalent to 
\begin{equation*}
a^{\theta }+b^{\theta }\geq e^{-1}.
\end{equation*}

Suppose $\theta <\infty $. Then this inequality will clearly hold once
either $a$ or $b$ is sufficiently close to $1$. Consequently, Theorem \ref%
{MainLarge}(ii) implies $D_{ab}\geq \theta $ if either $a$ or $b$ is
sufficiently large. Since $\theta $ was arbitrary, it follows that $D_{ab}$
tends to infinity as either $a$ or $b$ tend to $1$.

On the other hand, if $\theta >0$ and $a,b$ are both sufficiently small,
then $a^{\theta }+b^{\theta }<e^{-1}$ and hence $G(\theta )<\theta $.
Consequently, Theorem \ref{MainLarge}(i) implies that $D_{ab}<\theta $ and
hence $D_{ab}\rightarrow 0$ as both $a,b\rightarrow 0$.
\end{proof}
\end{example}

\subsubsection{ $p_n$ chosen from the two-element set $\{ p, 1 -p \}$ for
fixed $0 < p < 1/2$}

For our next example, we consider the deterministic Moran set $\mathcal{C}%
_{ab}$, but with $p_{n}$ chosen from a two-element set.

\begin{example}
\label{Cab} Consider the deterministic Moran set $\mathcal{C}_{ab}$ with $%
a<b $, but in this case let $\mu $ be the random measure with probability $%
p_{n}$ chosen uniformly from the two values $p$ or $1-p$ where $0<p<1/2$ is
fixed. Define $\beta $ and $\eta $ by $a^{\beta }=b$ and $(1-p)^{\eta }=p,$
so $\beta <1$ and $\eta >1$. We claim 
\begin{equation}
\overline{\dim }_{\Phi }\mu _{\omega }=\left\{ 
\begin{array}{cc}
\frac{\log p+\log (1-p)}{2\log b} & \text{if }\eta +1+\beta -3\eta \beta
\geq 0 \\ 
\frac{\log p}{\frac{1}{2}(\log a+\log b)} & \text{if }\eta +1+\beta -3\eta
\beta <0%
\end{array}%
\right. \text{ a.s}.  \label{Claim}
\end{equation}

\begin{proof}
Let $c(\theta )=a^{\theta }/(a^{\theta }+b^{\theta })$. Note that $c(\theta
) $ is a decreasing function and for $\theta \geq 0,$ $c(\theta )\leq 1/2$.
In particular, there is no non-negative solution to $c(\theta )=1-p>1/2$.
Let $\theta _{0}$ satisfy $c(\theta _{0})=p$, so 
\begin{equation*}
\theta _{0}=\frac{\log ((1-p)/p)}{\log (b/a)}=\frac{(\eta -1)\log (1-p)}{%
(1-\beta )\log a}.
\end{equation*}

If $\theta \geq \theta _{0},$ then both $p,1-p\geq c(\theta ),$ so $%
Y_{n}=\log (1-p_{n})$ and $Z_{n}=\log b$.

If $0\leq \theta <\theta _{0},$ then $p\leq c(\theta )<1-p,$ hence if $%
p_{n}=p,$ then $Y_{n}=\log p_{n}=\log p$ and $Z_{n}=\log a$, while if $%
p_{n}=1-p,$ then $Y_{n}=\log (1-p_{n})=\log p$ and $Z_{n}=\log b$.

It is easy to see from these observations that 
\begin{equation*}
\mathbb{E}(Y_{1})=\left\{ 
\begin{array}{cc}
\frac{1}{2}(\log p+\log (1-p)) & \text{if }\theta \geq \theta _{0} \\ 
\log p & \text{if }0\leq \theta <\theta _{0}%
\end{array}%
\right.
\end{equation*}%
and 
\begin{equation*}
\mathbb{E}(Z_{1})=\left\{ 
\begin{array}{cc}
\log b & \text{if }\theta \geq \theta _{0} \\ 
\frac{1}{2}(\log a+\log b) & \text{if }0\leq \theta <\theta _{0}%
\end{array}%
\right. .
\end{equation*}%
Hence 
\begin{equation*}
G(\theta )=\left\{ 
\begin{array}{cc}
\frac{\log p+\log (1-p)}{2\log b} & \text{if }\theta \geq \theta _{0} \\ 
\frac{\log p}{\frac{1}{2}(\log a+\log b)} & \text{if }0\leq \theta <\theta
_{0}%
\end{array}%
\right. .
\end{equation*}%
Replacing $p$ by $(1-p)^{\eta }$ and $b$ by $a^{\beta },$ this is the same
as stating%
\begin{equation*}
G(\theta )=\left\{ 
\begin{array}{cc}
\frac{(\eta +1)\log (1-p)}{2\beta \log a} & \text{if }\theta \geq \frac{%
(\eta -1)\log (1-p)}{(1-\beta )\log a} \\ 
\frac{\eta \log (1-p)}{\frac{1}{2}(\beta +1)\log a} & \text{if }0\leq \theta
<\frac{(\eta -1)\log (1-p)}{(1-\beta )\log a}%
\end{array}%
\right. .
\end{equation*}

It is easy to check that if $\eta +1+\beta -3\eta \beta <0,$ then 
\begin{equation*}
\frac{\eta \log (1-p)}{\frac{1}{2}(1+\beta )\log a}<\frac{(\eta -1)\log (1-p)%
}{(1-\beta )\log a},
\end{equation*}%
so $G(\alpha )=\alpha $ for%
\begin{equation*}
\alpha =\frac{\eta \log (1-p)}{\frac{1}{2}(\beta +1)\log a}=\frac{\log p}{%
\frac{1}{2}(\log a+\log b)}.
\end{equation*}%
If $\alpha <\psi <\theta _{0},$ then obviously $G(\psi )=\alpha <\psi $. If $%
\psi \geq \theta _{0}>\alpha ,$ then one can also check that $\eta +1+\beta
-3\eta \beta <0$ implies 
\begin{equation*}
\frac{(\eta +1)\log (1-p)}{2\beta \log a}<\frac{\eta \log (1-p)}{\frac{1}{2}%
(\beta +1)\log a},
\end{equation*}%
so again we have $G(\psi )<\alpha <\psi $.

Similarly, if $\eta +1+\beta -3\eta \beta \geq 0,$ then 
\begin{equation*}
\frac{(\eta +1)\log (1-p)}{2\beta \log a}\geq \frac{(\eta -1)\log (1-p)}{%
(1-\beta )\log a},
\end{equation*}%
hence $G(\alpha )=\alpha $ for 
\begin{equation*}
\alpha =\frac{(\eta +1)\log (1-p)}{2\beta \log a}=\frac{\log p+\log (1-p)}{%
2\log b}
\end{equation*}%
and if $\psi >\alpha ,$ then $G(\psi )=\alpha <\psi $.

It follows from the theorem that $\overline{\dim }_{\Phi }\mu $ is as
claimed in (\ref{Claim}).
\end{proof}
\end{example}

\subsubsection{ $p_n$ deterministic and $(a,b)$ random}

For our last two examples we now take $p_n = 1/2$ and choose the scalings $%
a_n$ and $b_n$ randomly.

\begin{example}
\label{exa,b} Consider the random Moran set, $\mathcal{C}_{\omega },$ where $%
a_{n},b_{n}$ are chosen independently from $\{A,B\}$ with equal likelihood
and $0<A<B<1/2$. Let $\mu _{\omega }$ be the random measure supported on $%
\mathcal{C}_{\omega }$ where $p_{n}=1/2$ for all $n$. Obviously, $%
Y_{n}(\theta )=\log 1/2$ for all $n$ and all $\theta $. Note that the
condition $p_{n}\leq a_{n}^{\theta }/(a_{n}^{\theta }+b_{n}^{\theta })$
simply reduces to the inequality $a_{n}\geq b_{n}$ and this is true whenever 
$a_{n}=A$ or $a_{n}=b_{n}=A$. Thus for all $n$ and $\theta ,$ 
\begin{eqnarray*}
\mathbb{E}(Z_{n}(\theta )) &=&\int_{\{a_{n}=B,a_{n}=b_{n}=A\}}\log
a_{n}+\int_{\{a_{n}=A,b_{n}=B\}}\log b_{n} \\
&=&\frac{1}{2}\log B+\frac{1}{4}\log A+\frac{1}{4}\log B=\frac{3}{4}\log B+%
\frac{1}{4}\log A.
\end{eqnarray*}%
It follows from the theorem that for all large $\Phi ,$ 
\begin{equation*}
\overline{\dim }_{\Phi }\mu _{\omega }=\frac{4\log 1/2}{3\log B+\log A}\text{
a.s.}
\end{equation*}
\end{example}

\begin{example}
\label{exUnif} Fix $0 < 2 A < B < 1$ and consider the random Moran set with $%
(a_n,b_n)$ chosen uniformly over the set 
\begin{equation*}
\Lambda = \{ (x,y) : A \le \min\{ x, y \} \le x+y \le B \}.
\end{equation*}
Let $\mu _{\omega}$ be the random measure where $p_{n}=1/2$ for all $n$.
Similarly to the previous example, for all $n$ and $\theta $, $Y_{n}(\theta
)=\log 1/2$ and 
\begin{equation*}
Z_{n}(\theta )=\left\{ 
\begin{array}{cc}
\log a_{n} & \text{if }a_{n}\geq b_{n} \\ 
\log b_{n} & \text{if }a_{n}<b_{n}%
\end{array}%
\right. .
\end{equation*}%
Thus 
\begin{eqnarray*}
\mathbb{E}(Z_{n}) &=&\int_{a_{n}\geq b_{n}}\log a_{n}+\int_{a_{n}<b_{n}}\log
b_{n} = 2 \int_{a_n \ge b_n} \log a_n \\
&=&\frac{4}{(B-2A)^2} \left( \int_{A}^{B/2}\left( \int_{A}^{x}\log x\,
dy\right) dx+\int_{B/2}^{B-A} \left( \int_{A}^{B-x}\log x\, dy\right) dx
\right) \\
&=&\frac{4}{(B-2A)^2} \left( \int_{A}^{B/2}(x-A)\log x\, dx+\int_{B/2}^{B-A}
(B-A-x) \log x\, dx \right) \\
&=& \frac{ 2 (B-A)^2 \log(B-A) + B^2 \log(2) + 2 A^2 \log(A) - B^2 \log(B) -
6 (B/2 - A)^2}{(B - 2 A)^2}.
\end{eqnarray*}%
Hence for all large $\Phi$, almost surely we have 
\begin{equation*}
\overline{\dim }_{\Phi }\mu _{\omega }= \frac{ (B-2A)^2 \log(1/2)}{ 2
(B-A)^2 \log(B-A) + B^2 \log(2) + 2 A^2 \log(A) - B^2 \log(B) - 6 (B/2 - A)^2%
}.
\end{equation*}
\end{example}

\medskip

\subsubsection{ Further remarks on $G(\protect\theta )$}

\label{subsubsec:G_theta}

For a fixed $A,B$ with $0<A<B<1$, set 
\begin{equation*}
\Lambda =\{(a,b,z):A\leq \min \{a,b\}\leq a+b\leq B,0\leq z\leq 1\}
\end{equation*}%
as our parameter space. Then each point $(a,b,p)\in \Lambda $ defines an
iterated function system with probabilities (IFSP). If this configuration
(scalings $a,b$ and probabilities $p,1-p$) is chosen at every level, the
resulting (deterministic) Moran set and measure are both self-similar. The
associated function $G(\theta )$ is piecewise constant with at most one
discontinuity. It is possible to show that the location of the discontinuity
cannot be between the two values of $G(\theta )$. Using this it is not
difficult to see that there is a unique solution to $G(\theta )=\theta $.

In terms of our random model we can identify this single IFSP with a
probability measure on $\Lambda $ which is a point-mass at the point $%
(a,b,p) $. If, instead, we take a probability measure on $\Lambda $ which is
a combination of $N$ point masses, this is identified with a finite
collection of different IFSPs from which to randomly choose at each level,
with the choice independent from level to level. This time the function $%
G(\theta )$ has at most $N$ points of discontinuity and hence at most a
finite number of solutions to $G(\theta )=\theta $. It would be very
interesting to know if it were possible to construct an explicit example
where $G(\theta )=\theta $ has no solutions; this would happen if a point of
discontinuity of $G(\theta )$ coincided with a jump in the value from $%
G(\psi )>\psi $ to $G(\psi )<\psi $. For a single IFSP this is not possible,
but it is unclear if this might be possible for a collection of IFSPs.

On the other hand if we begin with a probability measure $\eta $ on $\Lambda 
$ which is absolutely continuous with respect to Lebesgue measure, then $%
G(\theta )$ is a continuous function of $\theta $ and so Corollary \ref%
{Cor:Cont} applies. It is worth pausing for a moment to contemplate why this
is the case. For each fixed value of $\theta >0$, the set $\Lambda $ is
partitioned into the two regions 
\begin{equation*}
\{p\leq \frac{a^{\theta }}{a^{\theta }+b^{\theta }}\}\quad \mbox{ and }\quad
\{p>\frac{a^{\theta }}{a^{\theta }+b^{\theta }}\}
\end{equation*}%
and the boundary between these regions is a smooth function of $(a,b,p)$ and
also of $\theta $. The values of $Y(\theta )$ and $Z(\theta )$ depend
entirely on which of these two sets the particular (random) choice of $%
(a,b,p)$ belongs to, and thus the expected values of $Y$ and $Z$ are given
by the distribution of $\eta $ over these two sets. Since the boundary is a
smooth surface, if $\eta $ is absolutely continuous, changing $\theta $
moves the boundary smoothly and thus changes $G(\theta )$ in a continuous
way.

\subsection{Relating $\dim_{\Phi} \protect\mu$ to $\dim_{\Phi} {\mathcal{C}}_%
\protect\omega$}

\label{sec:Large_dim_underlying}


It is known that $\overline{\dim}_\Phi \mu \ge \overline{\dim}_\Phi \supp %
\mu $ for any measure $\mu$ and if $\mu$ is doubling then we also have $%
\underline{\dim}_\Phi \mu \le \underline{\dim}_\Phi \supp \mu$ (see \cite[%
Prop 2.9]{HH}). This leads us to ask if we can arrange for an almost sure
equality, that is, can we choose the $p_n(\omega)$ in such a way so that $%
\overline{\dim}_\Phi \mu_\omega = \overline{\dim}_\Phi {\mathcal{C}}_\omega$
for $\omega$ in a set of full measure.

There is a standard, and ``natural'', way of doing this for a single IFS
with probabilities. Given scaling factors $a$ and $b$, we set $p = a^d$ and $%
1 - p = b^d$, where $d>0$ is the solution to the Moran equation $a^x + b^x =
1$. This choice of $p$ will ``balance'' the scaling of the lengths with the
redistribution of the mass to ensure that $\dim \mu = \dim {\mathcal{C}}$,
with all the large $\Phi$-dimensions coinciding with the Hausdorff
dimension. In fact, in this particular case it is easy to see that $%
G^{\prime }(\theta) = G(\theta) = d$ for all $\theta$.

However, even in the next simplest case of randomly choosing between two
IFSPs this ``natural'' choice of probabilities does not typically give $%
\overline{\dim}_\Phi \mu_\omega = \overline{\dim}_\Phi {\mathcal{C}}_\omega$
almost surely. As an example, take the IFSP $\{ x/3, x/9 + 8/9 \}$ with
probabilities $p, 1-p$ and a second IFSP $\{ x/4, x/16 + 15/16\}$ with
corresponding probabilities $q, 1-q$. To get our random ${\mathcal{C}}%
_\omega $ and $\mu_\omega$ we will choose equally likely between these two
IFSs at each level.

The Moran equation for the first IFS is $3^{-x}+3^{-2x}=1$, whose solution
is 
\begin{equation*}
d_{1}=\frac{\ln \left( \frac{\sqrt{5}-1}{2}\right) }{-\ln (3)},%
\mbox{ with
corresponding }p=3^{-d_{1}}=\frac{\sqrt{5}-1}{2}.
\end{equation*}%
Similarly, the Moran equation for the second is $4^{-x}+4^{-2x}=1$, with
solution 
\begin{equation*}
d_{2}=\frac{\ln \left( \frac{\sqrt{5}-1}{2}\right) }{-\ln (4)},%
\mbox{ and
corresponding }q=4^{-d_{2}}=\frac{\sqrt{5}-1}{2}.
\end{equation*}%
Using these choices for $p$ and $q$, elementary computations show that 
\begin{equation*}
G(\theta )=%
\begin{cases}
\frac{2\ln \left( \frac{\sqrt{5}-1}{2}\right) }{\ln (\frac{1}{4})+\ln (\frac{%
1}{3})}, & \text{ if }\theta \leq \frac{\ln \left( \frac{\sqrt{5}-1}{2}%
\right) }{\ln (\frac{1}{4})}; \\[10pt] 
\frac{3\ln \left( \frac{\sqrt{5}-1}{2}\right) }{\ln (\frac{1}{4})+\ln (\frac{%
1}{9})}, & \text{ if }\frac{\ln \left( \frac{\sqrt{5}-1}{2}\right) }{\ln (%
\frac{1}{4})}<\theta \leq \frac{\ln \left( \frac{\sqrt{5}-1}{2}\right) }{\ln
(\frac{1}{3})}; \\[10pt] 
\frac{2\ln \left( \frac{\sqrt{5}-1}{2}\right) }{\ln (\frac{1}{4})+\ln (\frac{%
1}{3})}, & \text{ if }\theta >\frac{\ln \left( \frac{\sqrt{5}-1}{2}\right) }{%
\ln (\frac{1}{3})}.%
\end{cases}%
\end{equation*}%
From this we can deduce that 
\begin{equation*}
\overline{\dim }_{\Phi }\mu _{\omega }=\frac{3\ln \left( \frac{\sqrt{5}-1}{2}%
\right) }{\ln (\frac{1}{4})+\ln (\frac{1}{9})}\approx 0.402\text{ almost
surely.}
\end{equation*}

The almost sure Hausdorff dimension of ${\mathcal{C}}_\omega$ is given by
the solution $D > 0$ to $( 3^{-x} + 3^{-2 x})^{1/2}( 4^{-x} + 4^{-2
x})^{1/2} = 1$ (see \cite{Ham}). It is conjectured in \cite{Tr2} that for
all large $\Phi$ 
\begin{equation*}
\overline{\dim}_\Phi {\mathcal{C}}_\omega = \dim_H {\mathcal{C}}_\omega = D
\approx 0.388 \mbox{ almost surely.}
\end{equation*}

On the other hand, if there is a $d>0$ so that $a_n(\omega)^d +
b_n(\omega)^d = 1$ for all $\omega$ and $n$ (i.e., for all possible IFS in
the given model), then choosing $p_n(\omega) = a_n(\omega)^d$ will result in 
$d = \overline{\dim}_\Phi \mu_\omega = \overline{\dim}_\Phi {\mathcal{C}}%
_\omega$ almost surely. This is because in this very special situation we
will have $G(\theta) = d$ for all $\theta$. We conjecture that, other than
in this special case, generically the ``natural'' choice of $p_n$ will
result in $\overline{\dim}_\Phi \mu_\omega > \overline{\dim}_\Phi {\mathcal{C%
}}_\omega$ almost surely. Verifying this by explicit computations seems to
be exceedingly complicated.

\subsection{Comments on a more general construction}

\label{sec:multi}

In this short subsection we briefly indicate how we can modify our
construction so that it works in ${\mathbb{R}}^D$ and with the possibility
of more than two children per parent. We can also allow the number of
children at each level to be random and change from level to level. None of
these significantly change anything as long as the number of children is
uniformly bounded. To describe the generalization, we first need to
establish some notation and definitions.

For $I\subset {\mathbb{R}}^{d}$, we denote by $diam(I)$ the diameter of $I$.
Given $r>0$, we say $J\subseteq I$ is an $r$-\textbf{similarity} of $I$ if
there is a similarity $S$ such that $J=S(I)$ and $diam(J)=r\cdot diam(I)$. A
collection of $r_{j}$-similarities, $J_{1},J_{2},\ldots ,J_{k}$, (possibly
of distinct contraction factors) is $\tau $-\textbf{separated} if $%
d(J_{i},J_{j})\geq \tau \cdot diam(I)$ for all $i\neq j$. If such a
collection exists, we say that $I$ has the $(k,\tau )$-\textbf{separation
property}.

In the event that the interior of $I$ is non-empty, then for a given $k$ and
small enough $\tau > 0$, it is easy to see that $I$ will have the $(k,\tau)$%
-separation property for any $r_j \le \rho_k$, $j=1,2,\ldots, k$, for a
suitably small $\rho_k > 0$. For example, if $I = [0,1]$ and $\tau k < 1$,
then $\rho_k = (1 - (k-1)\tau)/k$ will work. We can view the $(k,\tau)$%
-separation condition as a uniform strong separation condition. 

Lemma \ref{L1}, which relates balls with level $n$ sets and thus contains
the essential geometric result, is changed very little in the more general
setup. We redefine $L$ by the condition that 
\begin{equation*}
2 B^{L-1} \le \tau
\end{equation*}
and replace $1-B$ with $\tau$ in the proof and everything else is the same.

Let $I_{0}$ be a fixed compact subset of ${\mathbb{R}}^{D}$ with non-empty
interior and diameter one. Fix $\tau \in (0,1)$, $K\geq 2$ and let $B_{i}\in
(0,1)$, $i=2,\ldots ,K$, be such that $I_{0}$ has the $(i,\tau )$-separation
property for all $r_{j}\leq B_{i}$. We again let $A\in (0,\min_{i}B_{i})$.
For each $\omega $ and step $n$ in the construction, we take the random
variables $K_{n}=k_{n}(\omega )\in \{2,3,\ldots ,K\}$ and $%
a_{n}^{(1)}(\omega ),...,a_{n}^{(K_{n})}(\omega )$ where $%
a_{n}^{(j)}=a_{n}^{(j)}(\omega )\geq A$ for each $j=1,2,\ldots ,K_{n}$ and
also $a_{n}^{(1)}+a_{n}^{(2)}+\cdots +a_{n}^{(K_{n})}\leq B_{K_{n}}$; these
determine the relative sizes of the children at step $n$. Specifically, the
children $J_{j}(\omega )$ of the parent $I_{n}(\omega )=I_{n}$ are $%
a_{n}^{(j)}$-similarities of $I_{n}$, for $j=1,2,\ldots ,K_{n}$, which are $%
\tau $-separated. The random Moran set $\mathcal{C}_{\omega }$ is then
defined (as usual) to be 
\begin{equation*}
\mathcal{C}_{\omega }=\bigcap_{n=1}^{\infty }{\mathcal{M}}_{n}(\omega ),
\end{equation*}%
where ${\mathcal{M}}_{n}(\omega )$ is the union of the step $n$ children.

Define a random measure $\mu _{\omega }$ supported on this Moran set $%
\mathcal{C}_{\omega }$ by the rule that if the children of $I_{n}$ are
labelled $I_{n}^{(j)}$, $j=1,...,K_{n}$, then $\mu _{\omega
}(I_{n}^{(j)})=p_{n}^{(j)}$ $\mu _{\omega }(I_{n}),$ where the random
variables $p_{n}^{(j)}(\omega )\geq 0$ satisfy $%
\sum_{j=1}^{K_{n}}p_{n}^{(j)}=1$ for all $n$. We assume that $\mathbb{E}%
(\left( a_{n}^{(j)}\right) ^{-t})<\infty $ and $\mathbb{E}(\left(
p_{n}^{(j)}\right) ^{-t})<\infty $ for some $t>0$ and all $j=1,...,K_{n}$
and $n$.

Define 
\begin{equation*}
\begin{array}{c}
Y_{n}(\theta )(\omega )=\log p_{n}^{(m)}(\omega ) \\ 
Z_{n}(\theta )(\omega )=\log a_{n}^{(m)}(\omega )%
\end{array}%
\text{ where }\frac{a_{n}^{(m)\theta }}{p_{n}^{(m)}}=\max \left( \frac{%
a_{n}^{(k)\theta }}{p_{n}^{(k)}}:k=1,...,K_{n}\right)
\end{equation*}%
and, as before, define 
\begin{equation*}
G(\theta )=\frac{\mathbb{E}_{\omega }(Y_{1}(\theta )(\omega ))}{\mathbb{E}%
_{\omega }(Z_{1}(\theta )(\omega ))}.
\end{equation*}

Essentially the same arguments as before show that Theorem \ref{MainLarge}
holds in this case as well.

\begin{example}
Suppose $K_n=3$ (the same for all $n$) and the ratios are $a_{n}^{(1)}=1/4,$ 
$a_{n}^{(2)}=a_{n}^{(3)}=1/16$ for all $\omega $. Assume the probabilities $%
p_{n}^{(j)}$ are $1/2,1/4,1/4$ with $1/2$ assigned to position $j$ with
equal likelihood. Note that if $p_{n}^{(1)}=1/2,$ then $\max \left(
a_{n}^{(k)\theta }/p_{n}^{(k)}\right) =a_{n}^{(1)\theta }/p_{n}^{(1)}$ if $%
\theta \geq 1/2$ and $a_{n}^{(2)\theta }/p_{n}^{(2)}$ otherwise. If $%
p_{n}^{(j)}=1/2$ for $j=2,3,$ then $\max \left( a_{n}^{(k)\theta
}/p_{n}^{(k)}\right) =a_{n}^{(1)\theta }/p_{n}^{(1)}$ for all $\theta \geq 0$%
. One can check that 
\begin{equation*}
\mathbb{E}(Y_{1}(\theta ))=\left\{ 
\begin{array}{cc}
\log (1/4) & \text{if }\theta <1/2 \\ 
\frac{5}{3}\log (1/2) & \text{ if }\theta \geq 1/2%
\end{array}%
\right.
\end{equation*}%
and 
\begin{equation*}
\mathbb{E}(Z_{1}(\theta ))=\left\{ 
\begin{array}{cc}
\frac{4}{3}\log (1/4) & \text{if }\theta <1/2 \\ 
\log (1/4) & \text{ if }\theta \geq 1/2%
\end{array}%
\right. .
\end{equation*}%
Thus 
\begin{equation*}
G\mathbb{(}\theta )=\left\{ 
\begin{array}{cc}
3/4 & \text{if }\theta <1/2 \\ 
5/6 & \text{ if }\theta \geq 1/2%
\end{array}%
\right.
\end{equation*}%
and consequently, for all large dimension functions $\Phi $, $\overline{\dim}
_{\Phi}\mu =5/6$ a.s.
\end{example}

\section{\protect\bigskip Dimension results for small $\Phi $}

\label{sec:SmallPhi}

We now move to a discussion of the \textquotedblleft
small\textquotedblright\ dimension functions $\Phi $. Recall that this means
that $\Phi \ll \log \left\vert \log t\right\vert /\left\vert \log
t\right\vert $. We again restrict our discussion to the case of two children
per parent interval for the sake of clarity. The modifications necessary for
the more general case are straightforward.

\subsection{The Dimension Theorem for Small $\Phi $}

\label{sec:dimthm_small}

Put 
\begin{eqnarray*}
\alpha &=&\max \left( \esssup\left( \frac{\log p_{1}(\omega )}{\log
a_{1}(\omega )}\right) ,\esssup\left( \frac{\log (1-p_{1}(\omega ))}{\log
b_{1}(\omega )}\right) \right) , \\
\beta &=&\min \left( \essinf\left( \frac{\log p_{1}(\omega )}{\log
a_{1}(\omega )}\right) ,\essinf\left( \frac{\log (1-p_{1}(\omega ))}{\log
b_{1}(\omega )}\right) \right) .
\end{eqnarray*}

\begin{theorem}
\label{small}There is a set $\Gamma $ of full measure, such that $\overline{%
\dim }_{\Phi }\mu _{\omega }=\alpha $ and $\underline{\dim }_{\Phi }\mu
_{\omega }=\beta $ for all $\omega \in \Gamma $ and for all small dimension
functions $\Phi $.
\end{theorem}

\begin{proof}
We will begin by verifying that $\overline{\dim }_{\Phi }\mu \leq \alpha $
a.s. (and this will hold for all choices of $\Phi ,$ not just small $\Phi $%
). Of course, this is obvious if $\alpha =\infty $. Otherwise, consider the
Moran interval $I_{v}(\omega )=I_{v_{1}...v_{N}}$ and descendent interval $%
I_{u}(\omega )=I_{v_{1}...v_{n}}$ where $\left\vert I_{u}\right\vert \leq
\left\vert I_{v}\right\vert ^{1+\Phi (\left\vert I_{v}\right\vert )}$. Then%
\begin{equation*}
\frac{\mu _{\omega }(I_{v})}{\mu _{\omega }(I_{u})}=\prod_{\substack{ j=N+1 
\\ v_{j}=0}}^{n}p_{j}^{-1}\prod_{\substack{ j=N+1  \\ v_{j}=1}}%
^{n}(1-p_{j})^{-1}
\end{equation*}%
and 
\begin{equation*}
\frac{\left\vert I_{v}\right\vert }{\left\vert I_{u}\right\vert }=\prod 
_{\substack{ j=N+1  \\ v_{j}=0}}^{n}a_{j}^{-1}\prod_{\substack{ j=N+1  \\ %
v_{j}=1}}^{n}b_{j}^{-1},
\end{equation*}%
so 
\begin{equation*}
\frac{\frac{\mu _{\omega }(I_{v})}{\mu _{\omega }(I_{u})}}{\left( \frac{%
\left\vert I_{v}\right\vert }{\left\vert I_{u}\right\vert }\right) ^{\alpha }%
}=\prod_{\substack{ j=N+1  \\ v_{j}=0}}^{n}\frac{p_{j}^{-1}}{a_{j}^{-\alpha }%
}\prod_{\substack{ j=N+1  \\ v_{j}=1}}^{n}\frac{(1-p_{j})^{-1}}{%
b_{j}^{-\alpha }}.
\end{equation*}

Almost surely, $\alpha \geq \log p_{j}/\log a_{j}$ and $\alpha \geq \log
(1-p_{j})/\log b_{j}$ for all $j,$ hence $a_{j}^{-\alpha }\geq p_{j}^{-1}$
and $b_{j}^{-\alpha }\geq (1-p_{j})^{-1}$ a.s. Thus 
\begin{equation*}
\frac{\mu _{\omega }(I_{v})}{\mu _{\omega }(I_{u})}\leq \left( \frac{%
\left\vert I_{v}\right\vert }{\left\vert I_{u}\right\vert }\right) ^{\alpha }%
\text{ a.s.}
\end{equation*}%
and consequently $\ \overline{\dim }_{\Phi }\mu \leq \alpha $ a.s.

\smallskip

For the reverse inequality, first suppose $\alpha <\infty $. Fix $i\in 
\mathbb{N}$. Without loss of generality, we will assume 
\begin{equation*}
\alpha =\esssup \left( \frac{\log p_{1}}{\log a_{1}}\right) =\esssup \left( 
\frac{\log p_{1}^{-1}}{\log a_{1}^{-1}}\right) .
\end{equation*}%
From the definition of the essential supremum, there must be some $0<\delta
_{i}<1$ such that 
\begin{equation*}
\mathcal{P}\left( \omega :\frac{\log p_{1}^{-1}}{\log a_{1}^{-1}}\geq \alpha
-\frac{1}{2i}\right) \geq \delta _{i}\text{.}
\end{equation*}

Let 
\begin{equation*}
J_{i}=\frac{|\log B|}{2|\log \delta _{i}|}\text{ and }\Phi _{i}(t)=\frac{%
J_{i}\log |\log t|}{|\log t|}.
\end{equation*}%
For each positive integer $N,$ let 
\begin{equation*}
\chi _{N,i}^{{}}=\frac{J_{i}\log (N|\log A|)}{|\log B|}.
\end{equation*}%
Clearly, $\chi _{N,i}^{{}}\rightarrow \infty $ as $N\rightarrow \infty $ and
the definitions ensure that $\delta _{i}^{\chi _{N,i}}\geq 1/N$ for large
enough $N$. Set 
\begin{equation*}
\Gamma _{N,i}=\left\{ \omega :\frac{\log p_{j}^{-1}}{\log a_{j}^{-1}}\geq
\alpha -\frac{1}{2i}\text{ for }j=N+1,...,N+\chi _{N,i}\right\} .
\end{equation*}%
As the tuples $(p_{n},a_{n},b_{n})$ are independent, 
\begin{equation*}
\mathcal{P}\left( \Gamma _{N,i}\right) =\prod_{j=N+1}^{N+\chi _{N,i}}%
\mathcal{P}\left( \omega :\frac{\log p_{j}^{-1}}{\log a_{j}^{-1}}\geq \alpha
-\frac{1}{2i}\right) =\delta _{i}^{\chi _{N,i}}\geq \frac{1}{N}\text{.}
\end{equation*}%
Thus if we let $N_{k}=k\log k,$ then for some suitably large $K_{0},$ 
\begin{equation*}
\sum_{k}\mathcal{P}(\Gamma _{N_{k},i}^{{}})\geq \sum_{k\geq K_{0}}\frac{1}{%
k\log k}=\infty .
\end{equation*}

As we can replace $\delta _{i}$ with any smaller, strictly positive number,
there is no loss of generality in assuming it is so small that $%
N_{k+1}>N_{k}+\chi _{N_{k},i}$. Hence the events $\Gamma _{N_{k},i}$ are
independent and thus the Borel-Cantelli lemma implies that $\mathcal{P}%
(\Gamma _{N_{k},i}^{{}}$ i.o.$)=1$ for each (fixed) $i$. Let $\Gamma _{i}$
be this set of full measure.

Take any $\omega \in \Gamma _{i}$ and consider any Moran interval, $%
I_{N}(\omega ),$ of step $N$. Let $I_{n}(\omega )$ be the left-most
descendent of $I_{N}$ at level $n=N+$ $\chi _{N,i}^{{}}$ (where we make the
choice of the left descendent since $\alpha =\esssup$ $\log p_{1}/\log a_{1}$%
). Since $\left\vert I_{N}\right\vert \geq A^{N}$ and the function $t^{\Phi
_{i}(t)}$ decreases as $t$ decreases to $0,$ the choice of $\chi _{N,i}^{{}}$
ensures 
\begin{equation*}
\left\vert I_{N}\right\vert ^{\Phi _{i}\left( \left\vert I_{N}\right\vert
\right) }\geq A^{N\Phi _{i}(A^{N})}=A^{\frac{J_{i}\log (N\left\vert \log
A\right\vert )}{\left\vert \log A\right\vert }}\geq B^{\chi _{N,i}^{{}}}\geq 
\frac{\left\vert I_{n}\right\vert }{\left\vert I_{N}\right\vert }.
\end{equation*}%
Hence $\left\vert I_{n}\right\vert \leq \left\vert I_{N}\right\vert ^{1+\Phi
_{i}\left( \left\vert I_{N}\right\vert \right) }$. As $\omega \in \Gamma
_{i},$ it follows that for infinitely many $N,$%
\begin{equation*}
\log p_{j}^{-1}\geq (\log a_{j}^{-1})(\alpha -\frac{1}{2i})\text{ for }%
j=N+1,...,N+\chi _{N,i},
\end{equation*}%
equivalently, 
\begin{equation*}
p_{j}^{-1}\geq a_{j}^{-(\alpha -1/(2i))}.
\end{equation*}%
Thus 
\begin{equation*}
\frac{p_{j}^{-1}}{a_{j}^{-(\alpha -1/i)}}\geq a_{j}^{-1/(2i)}\geq
B^{-(1/(2i))}\text{ for }j=N+1,...N+\chi _{N,i}\text{.}
\end{equation*}%
Consequently, for each fixed $i,$%
\begin{equation*}
\frac{\frac{\mu _{\omega }(I_{N})}{\mu _{\omega }(I_{n})}}{\left( \frac{%
\left\vert I_{N}\right\vert }{\left\vert I_{n}\right\vert }\right) ^{\alpha
-1/i}}=\prod_{j=N+1}^{N+\chi _{N,i}^{{}}}\frac{p_{j}^{-1}}{a_{j}^{-(\alpha
-1/i)}}\geq (B^{-1/(2i)})^{\chi _{N,i}}\text{.}
\end{equation*}%
and since $(B^{-1/(2i)})^{\chi _{N,i}}$ as $N\rightarrow \infty $ it follows
that there can be no constant $C$ such that 
\begin{equation*}
\frac{\mu _{\omega }(I_{N})}{\mu _{\omega }(I_{n})}\leq C\left( \frac{%
\left\vert I_{N}\right\vert }{\left\vert I_{n}\right\vert }\right) ^{\alpha
-1/i}
\end{equation*}%
for all such $N,n$. By Lemma \ref{lem:intervals_give_dimension} that implies 
$\overline{\dim }_{\Phi _{i}}\mu _{\omega }\geq \alpha -1/i$ for all $\omega
\in \Gamma _{i}$.

Let $\Gamma =\bigcap_{i=1}^{\infty }\Gamma _{i}$, a set of full measure, and
assume $\Phi $ is any small dimension function. Then there is some function $%
H(t)\rightarrow 0$ as $t\rightarrow 0$ so that 
\begin{equation*}
\Phi (t)\leq \frac{H(t)\log |\log t|}{|\log t|}\text{ for all }t\leq t_{0}.
\end{equation*}

Consequently, for each $i$ there is some $t_{i}>0$ such that $\Phi (t)\leq
\Phi _{i}(t)$ for all $t\leq t_{i}$. This property and our observations
above ensure that $\overline{\dim }_{\Phi }\mu _{\omega }\geq $ $\overline{%
\dim }_{\Phi _{i}}\mu _{\omega }\geq \alpha -1/i$ for all $i$ and all $%
\omega \in \Gamma $. We conclude that $\overline{\dim }_{\Phi }\mu _{\omega
}\geq \alpha $ for all $\omega \in \Gamma $, as we desired to show.

If $\alpha =\infty $, replacing `$\alpha -1/(2i)$' in the arguments with `$%
2i $'$,$ in the same manner we deduce that for every $i\in \mathbb{N}$ and
infinitely many $N,$ 
\begin{equation*}
\frac{\frac{\mu _{\omega }(I_{N})}{\mu _{\omega }(I_{n})}}{\left( \frac{%
\left\vert I_{N}\right\vert }{\left\vert I_{n}\right\vert }\right) ^{i}}\geq
B^{-i\chi _{N,i}}
\end{equation*}%
and, of course, this tends to infinity as $N\rightarrow \infty $. It follows
that $\overline{\dim }_{\Phi _{i}}\mu _{\omega }\geq i$ for all $\omega \in
\Gamma _{i}$, a set of full measure and with similar reasoning to above, we
deduce that $\overline{\dim }_{\Phi }\mu _{\omega }=\infty $ a.s.

If, instead, $\alpha = \esssup \left( \log (1-p_{1})/\log b_{1}\right) ,$ we
consider a Moran interval of level $N$ and its right-most descendent at
level $N+\chi _{N,i}^{{}},$ and argue in a similar fashion.

\smallskip

The arguments to establish $\underline{\dim }_{\Phi }\mu = \beta $ a.s. are
analogous and left to the reader.
\end{proof}

\begin{corollary}
Almost surely, $\dim _{A}\mu _{\omega }=\alpha $ and $\dim _{L}\mu _{\omega
}=\beta $.
\end{corollary}

\begin{proof}
This is immediate from the theorem as the constant function $\Phi =0$ is a
small dimension function.
\end{proof}

Let 
\begin{eqnarray*}
a_{0} &=&\essinf a_{1}(\omega )\text{, }b_{0}=\essinf b_{1}(\omega )\text{, }
\\
A_{0} &=&\esssup a_{1}(\omega )\text{, }B_{0}=\esssup b_{1}(\omega ),
\end{eqnarray*}%
\begin{eqnarray*}
p_{0} &=&\esssup p_{1}(\omega )\text{, }q_{0}=\esssup(1-p_{1}(\omega ))\text{%
,} \\
\text{ }P_{0} &=&\essinf p_{1}(\omega )\text{, }Q_{0}=\essinf %
(1-p_{1}(\omega )).
\end{eqnarray*}

and put 
\begin{equation*}
\alpha ^{\prime }=\max \left( \frac{\log P_{0}}{\log A_{0}},\frac{\log Q_{0}%
}{\log B_{0}}\right) \text{, }\beta ^{\prime }=\min \left( \frac{\log p_{0}}{%
\log a_{0}},\frac{\log q_{0}}{\log b_{0}}\right) .
\end{equation*}%
Notice that if $p_{n}$ is chosen independently of $(a_{n},b_{n})$ for all $n$%
, then $\alpha ^{\prime }=\alpha $ and $\beta ^{\prime }=\beta $.
Consequently, another immediate consequence of the theorem is

\begin{corollary}
Suppose $p_{n}$ is chosen independently of $(a_{n},b_{n})$ for all $n$.
There is a set $\Gamma $ of full measure, such that $\overline{\dim }_{\Phi
}\mu _{\omega }=\alpha ^{\prime }$ and $\underline{\dim }_{\Phi }\mu
_{\omega }=\beta ^{\prime }$ for all $\omega \in \Gamma $ and for all small
dimension functions $\Phi $.
\end{corollary}

\begin{example}
Consider, again, the Moran set $\mathcal{C}_{ab}$ and random measure $\mu
_{\omega }$ with probabilities chosen with equal likelihood from $\{p,1-p\}$
with $a<b$ and $p<1-p,$ as in Example \ref{Cab}. Then $\overline{\dim }%
_{\Phi }\mu _{\omega }=\log p/\log b$ and \underline{$\dim $}$_{\Phi }\mu
_{\omega }=\log (1-p)/\log a$ almost surely.
\end{example}

\begin{example}
In both examples \ref{exa,b} and \ref{exUnif}, it is trivial to compute the $%
\Phi $-dimensions for small $\Phi $. In both cases $a_{0}=b_{0}=A$, $%
A_{0}=B_{0}=B$ and $p_{0}=q_{0}=P_{0}=Q_{0}=1/2$. Thus $\overline{\dim }%
_{\Phi }\mu_{\omega} =\frac{\log 1/2}{\log B}$ and \underline{$\dim $}$%
_{\Phi }\mu_{\omega} =\frac{\log 1/2}{\log A}$ almost surely.
\end{example}

\begin{remark}
As in Subsection \ref{sec:multi}, suppose that each parent interval in the
Moran set construction has $K\geq 2$ children and define a random Moran set $%
\mathcal{C}_{\omega }$ and measure $\mu _{\omega }$ as was done there (with
the same assumptions). With the notation of that subsection, for $j=1,...,K$
put

\begin{eqnarray*}
\alpha &=&\max_{j=1,...,K}\esssup\left( \frac{\log p_{j}(\omega )}{\log
a_{j}(\omega )}\right) , \\
\beta &=&\min_{j=1,...,K}\essinf\left( \frac{\log p_{j}(\omega )}{\log
a_{j}(\omega )}\right) .
\end{eqnarray*}%
The same reasoning as in the proof of the theorem shows $\overline{\dim }%
_{\Phi }\mu _{\omega }=\alpha $ and $\underline{\dim }_{\Phi }\mu _{\omega
}=\beta $ for almost all $\omega $ and for all small dimension functions $%
\Phi $.
\end{remark}

\subsection{Relating $\dim_{\Phi} \protect\mu$ to $\dim_{\Phi} {\mathcal{C}}_%
\protect\omega$}

\label{sec:small_dim_underlying}


As in the case of the large dimension functions, it is natural to ask if one
could obtain the almost sure $\Phi $-dimensions of these random Moran sets
as the almost sure $\Phi $-dimensions of the random measures arising from 
\textit{some} choice of probabilities, as is the case when $a_{n}(\omega
)=b_{n}(\omega )$ for all $n$ and $\omega $ (see \cite{HM}). The following
example shows that this need not be the case for the small dimension
functions $\Phi $ when the random set is not generated by equicontractive
similarities.

\begin{example}
Choose $0<a_{2}<a_{1}<1/2$, \thinspace $0<b_{1}<b_{2}<1/2$ and consider the
family of random Moran sets $C_{\omega }$ where we choose $(a_{n},b_{n})$
independently and with equal likelihood from $\{(a_{1},b_{1}),$ $%
(a_{2},b_{2})\}$. We will let $\mathcal{C}_{1},\mathcal{C}_{2}$ denote the
(deterministic) Moran sets generated by $(a_{1},b_{1})$ and $(a_{2},b_{2})$
respectively. It is known \cite[Thm. 2.6]{FMT} that 
\begin{equation*}
\dim _{A}\mathcal{C}_{\omega }=\max \left\{ \dim _{A}\mathcal{C}_{1},\dim
_{A}\mathcal{C}_{2}\right\} \text{ a.s.}
\end{equation*}%
and that $\dim _{A}\mathcal{C}_{j}$ is the value of $d_{j}$ satisfying $%
a_{j}^{d_{j}}+b_{j}^{d_{j}}=1$ for $j=1,2$.

Let $0\leq p\leq q\leq 1$ and for convenience let%
\begin{equation*}
\lambda (p,q)=\max \left\{ \frac{\log p}{\log a_{1}},\frac{\log (1-q)}{\log
b_{2}}\right\} .
\end{equation*}%
Theorem \ref{small} shows that if we denote by $\mu _{\omega }=\mu _{\omega
}(p,q)$ the random measures supported on the Moran sets $\mathcal{C}_{\omega
}$ where we choose probabilities $p_{n}$ independently and with equal
likelihood from $\{p,q\}$, then for each $p,q$%
\begin{equation*}
\dim _{A}\mu _{\omega }(p,q)=\lambda (p,q)\text{ a.s.}
\end{equation*}%
Notice that a compactness argument ensures that $\inf_{0\leq p\leq q\leq
1}\lambda (p,q)=$ $\lambda (p_{0},q_{0})$ for a suitable choice of $%
p_{0},q_{0}$.

We will see that 
\begin{equation}
\dim _{A}\mathcal{C}_{\omega }<\lambda (p_{0},q_{0})\text{ a.s.}  \label{*}
\end{equation}

To prove this, we first note that at the minimal value of $\lambda (p,q)$ we
must have $\log p_{0}/\log a_{1}=\log (1-q_{0})/\log b_{2}$. Moreover, as
the function $\log p/\log a_{1}$ decreases as $p$ increases and the function 
$\log (1-q)/\log b_{2}$ decreases as $q$ decreases, the minimum value occurs
when $p_{0}=q_{0}$, hence at a choice of $p_{0}$ where 
\begin{equation*}
\frac{\log p_{0}}{\log a_{1}}=\frac{\log (1-p_{0})}{\log b_{2}}.
\end{equation*}%
If we suppose $\gamma $ is chosen so that $b_{2}=a_{1}^{\gamma },$ solving
the equation above gives that $p_{0}$ satisfies $p_{0}^{\gamma }+p_{0}=1$
or, equivalently, $p_{0}=(1-p_{0})^{1/\gamma }$. To summarize, 
\begin{equation*}
\lambda (p_{0},p_{0})=\frac{\log p_{0}}{\log a_{1}}=\frac{\log (1-p_{0})}{%
\log b_{2}}\text{ where }p_{0}^{\gamma }+p_{0}=1.
\end{equation*}

Now assume $\kappa $ is chosen so that $b_{1}=a_{1}^{\kappa }$. As $%
b_{1}<b_{2}$, we must have $\kappa >\gamma $. Furthermore, $d_{1}=\dim _{A}%
\mathcal{C}_{1}$ is defined by the rule $1=a_{1}^{d_{1}}+(a_{1}^{d_{1}})^{%
\kappa }$. Since $\kappa >\gamma ,$ if $a_{1}^{d_{1}}\leq p_{0}$ we obtain
the contradiction 
\begin{equation*}
1=a_{1}^{d_{1}}+(a_{1}^{d_{1}})^{\kappa }\leq p_{0}+p_{0}^{\kappa
}<p_{0}+p_{0}^{\gamma }=1.
\end{equation*}%
Hence $a_{1}^{d_{1}}>p_{0}$ and thus 
\begin{equation*}
\dim _{A}\mathcal{C}_{1}=d_{1}<\frac{\log p_{0}}{\log a_{1}}.
\end{equation*}

Likewise, if we assume $a_{2}^{\eta }=b_{2},$ then as $a_{1}>a_{2}$ we must
have $1/\eta >1/\gamma $. And as $d_{2}$ is defined by the rule $1=\left(
b_{2}^{d_{2}}\right) ^{1/\eta }+b_{2}^{d_{2}},$ it similarly follows that $%
b_{2}^{d_{2}}>1-p_{0},$ that is 
\begin{equation*}
\dim _{A}\mathcal{C}_{2}=d_{2}<\frac{\log (1-p_{0})}{\log b_{2}}.
\end{equation*}%
These observations prove (\ref{*}).

Since $\dim _{A}\mu _{\omega }=\overline{\dim }_{\Phi }\mu _{\omega }$ a.s.
for all small $\Phi ,$ and $\dim _{A}\mathcal{C}_{\omega }\geq \overline{%
\dim }_{\Phi }\mathcal{C}_{\omega }$ for all $\omega $ and $\Phi ,$ this
also proves that for all small dimension functions $\Phi $ and for each $p,q$
\begin{equation*}
\dim _{\Phi }\mathcal{C}_{\omega }<\frac{\log p_{0}}{\log a_{1}}\leq 
\overline{\dim }_{\Phi }\mu _{\omega }(p,q)\text{ a.s.}
\end{equation*}%
where $p_{0}$ is given by the rule $p_{0}+p_{0}^{\gamma }=1$ and $\gamma
=\log b_{2}/\log a_{1}$.

For an explicit example, suppose $a_{1}=1/2$, $b_{1}=1/4$, $a_{2}=1/3=b_{2}$%
. Then 
\begin{equation*}
\dim _{A}\mathcal{C}_{1}=\frac{\log (\left( \sqrt{5}-1\right) /2)}{\log 1/2}%
\approx .69\text{ and }\dim _{A}\mathcal{C}_{2}=\frac{\log 2}{\log 3}\approx
.63
\end{equation*}%
so 
\begin{equation*}
\dim _{A}\mathcal{C}_{\omega }=\dim _{A}\mathcal{C}_{1}\approx .69\text{ a.s.%
}
\end{equation*}%
We have $b_{2}=a_{1}^{\gamma }$ for $\gamma =\log 3/\log 2$ and Maple gives
the approximate solution to $p_{0}^{\gamma }+p_{0}=1$ as $p_{0}\approx .58.$
Hence for each choice of $p,q,$ 
\begin{equation*}
\dim _{A}\mu _{\omega }(p,q)\geq \frac{\log p_{0}}{\log 1/2} \approx .78%
\text{ a.s.}
\end{equation*}
\end{example}

\section*{Appendix: Examples of $G(\protect\theta)$}

Even though the function $G(\theta)$ is only used as a technical tool for
proving our results, it is helpful to see a few plots of some examples.
Referring to Figure \ref{fig:Gtheta}, we present our four examples starting
with the first row and going left to right. We comment that we used Maple
for both the computations and for the plots.

For the first example, we set $b=2a$ and choose $a$ uniformly from the
interval $[\frac{1}{10},\frac{3}{10}]$. In addition, $p$ is chosen uniformly
from $[0,1]$. The resulting function $G(\theta )$ is clearly smooth, but not
monotone.

In the second, we now set $b = 50 a$ and choose $a$ uniformly in the range $%
1/100 \le a \le (1 - 1/100)/51$. Again $p$ is chosen uniformly from $[0,1]$.
This time $G(\theta)$ is monotone increasing, but still smooth.

For our third example, we again set $b = 2 a$ and choose $a$ uniformly from $%
[\frac{1}{10}, \frac{3}{10}]$. However, this time $p$ is chosen uniformly
from the set $[0,\frac{1}{10}] \cup [ \frac{1}{5}, \frac{2}{5}] \cup [\frac{7%
}{10}, 1]$. The function is continuous, but only piecewise smooth and not
monotone. \emph{Note that the vertical axis is different than in the other
plots.} This was done in order to more clearly show the shape of the graph.

Finally, for our fourth example we take the 9 triples $(a,b,p)$ 
\begin{align*}
\left( \frac{2}{25}, \frac{18}{25}, \frac{1}{18} \right), \left( \frac{4}{25}%
, \frac{16}{25}, \frac{1}{8} \right), \left( \frac{6}{25}, \frac{14}{25}, 
\frac{3}{14} \right), \left( \frac{8}{25}, \frac{12}{25}, \frac{1}{3}
\right), \left( \frac{10}{25}, \frac{10}{25}, \frac{1}{2} \right), \\
\left( \frac{12}{25}, \frac{8}{25}, \frac{5}{6} \right), \left( \frac{14}{25}%
, \frac{6}{25}, \frac{5}{7} \right), \left( \frac{16}{25}, \frac{4}{25}, 
\frac{5}{8} \right), \left( \frac{18}{25}, \frac{2}{25}, \frac{5}{9}
\right), \hskip 1 cm
\end{align*}
which each define an IFS with probabilities. For our random model, at each
level we choose one of these IFSPs equally likely. As is clear from the
plot, in this case $G(\theta)$ is discontinuous, piecewise constant, and
non-monotone. The discontinuities occur at $\theta = \log(p/(1-p))/\log(a/b)$
for the various choices of $a$, $b$ and $p$.

\begin{figure}[tbp]
\includegraphics[width=2 in]{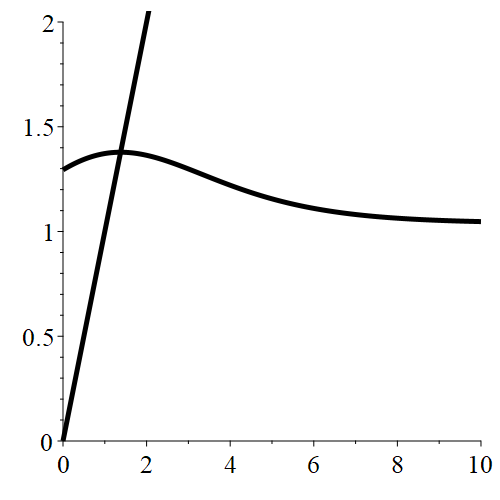} 
\includegraphics[width=2
in]{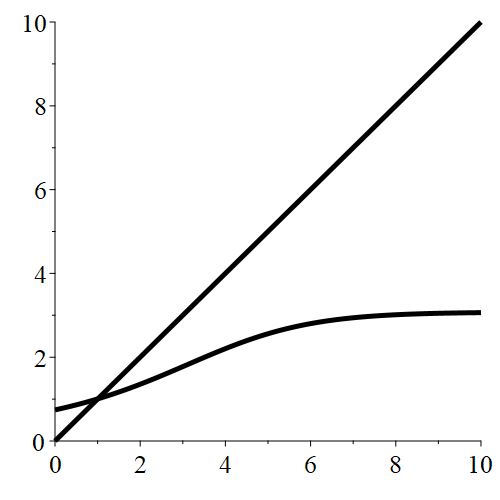}
\par
\includegraphics[width=2 in, height = 2 in]{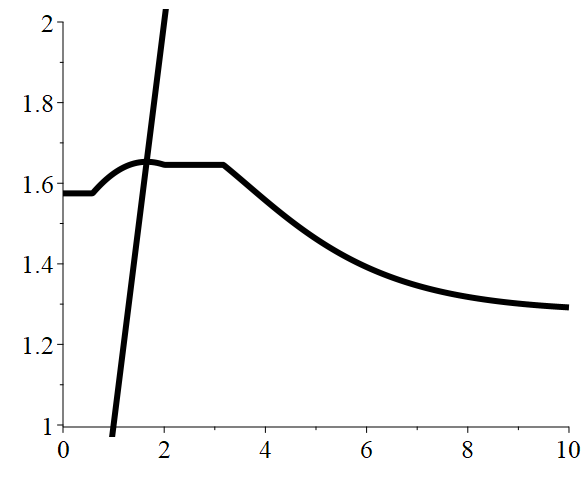} %
\includegraphics[width=2 in]{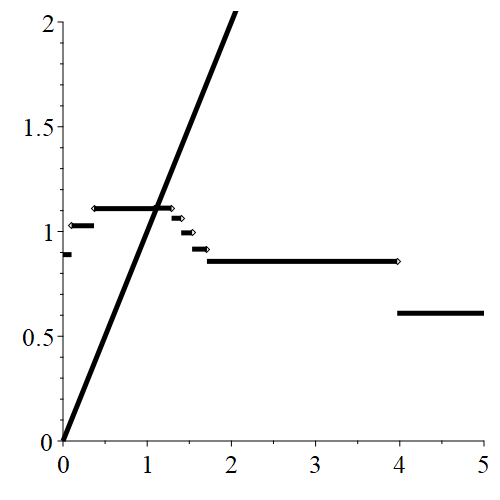}
\caption{Four different examples of $G(\protect\theta)$, described in the
appendix.}
\label{fig:Gtheta}
\end{figure}

\end{document}